\documentclass{article}
\usepackage{geometry}
\usepackage{amsmath,amsthm,amssymb,mathtools,microtype,needspace}
\usepackage{mathrsfs} 
\usepackage{hyperref}
\hypersetup{
    pdftitle={Sharp Hardy inequalities for the Vladimirov--Taibleson operator on the n-dimensional p-adic space},
    pdfauthor={},
    pdfsubject={Sharp Lq constants, ground state representations, and nonlinear difference energies},
    pdfkeywords={p-adic analysis, Vladimirov-Taibleson operator, Hardy inequality, sharp constant}
}

\theoremstyle{plain}
\newtheorem{theorem}{Theorem}[section]
\newtheorem{lemma}[theorem]{Lemma}
\newtheorem{corollary}[theorem]{Corollary}
\newtheorem{proposition}[theorem]{Proposition}

\theoremstyle{definition}
\newtheorem{definition}{Definition}[section]

\newtheoremstyle{boldremark}
  {\topsep}{\topsep}
  {\normalfont}{}
  {\bfseries}{.}{0.5em}{}
\theoremstyle{boldremark}
\newtheorem{remark}[theorem]{Remark}

\numberwithin{equation}{section}
\allowdisplaybreaks[2]

\newcommand{\bx}{\boldsymbol{x}}
\newcommand{\by}{\boldsymbol{y}}
\newcommand{\bz}{\boldsymbol{z}}
\newcommand{\bxi}{\boldsymbol{\xi}}
\newcommand{\Qp}{\mathbb{Q}_p}
\newcommand{\Qpn}{\mathbb{Q}_p^n}
\newcommand{\Zp}{\mathbb{Z}_p}
\newcommand{\Sspace}{\mathcal{S}(\Qpn)}
\newcommand{\Sp}{\mathcal{S}'(\Qpn)}
\newcommand{\rn}[1]{\lVert #1\rVert_p}

\newcommand{\dd}{\,d}
\newcommand{\ind}{\boldsymbol{1}}
\newcommand{\GammaN}{\Gamma_p^{(n)}}
\newcommand{\E}{\mathcal{E}}
\newcommand{\G}{\mathcal{G}}
\newcommand{\W}{\mathcal{W}}
\newcommand{\F}{\mathcal{F}}
\newcommand{\supp}{\operatorname{supp}}
\newcommand{\csharp}{C^{\sharp}_{p,n,\alpha}}
\newcommand{\Msharp}{M_{p,n,s,q}}
\newcommand{\Hsharp}{H_{p,n,s,q}}
\newcommand{\X}{\mathcal{X}_{s,q}}
\newcommand{\doi}[1]{\href{https://doi.org/#1}{\nolinkurl{doi:#1}}}
\newcommand{\arxiv}[1]{\href{https://arxiv.org/abs/#1}{\nolinkurl{arXiv:#1}}}

\title{Sharp Hardy inequalities for the Vladimirov--Taibleson operator on $\Qpn$}

\author{Yaojia Sun}
\date{}

\begin{document}
\maketitle

\begin{abstract}
We prove sharp Hardy inequalities on $\Qpn$ for the quadratic form of the Vladimirov--Taibleson operator, the $L^q$ norm of its fractional derivative, and a nonlinear difference energy. Under the self-dual Fourier normalization, the sharp quadratic-form constant for $0<\alpha<n$ is
$[(1-p^{-(n-\alpha)/2})/(1-p^{-(n+\alpha)/2})]^2$, with a remainder given by the ground state representation.
For $1<q<\infty$, $q'=q/(q-1)$, and $0<s<n/q$, a weighted Schur estimate determines the exact $L^q$ norm of the weighted Riesz potential $\rn{\bx}^{-s}I_s$, giving the sharp linear Hardy constant
$[\GammaN(n/q'+s)/\GammaN(n/q')]^q$. The two-point Picone inequality and shell decomposition give the sharp constant
$2(1-p^{-n})\sum_{m\ge1}p^{-sqm}(1-p^{-(n/q-s)m})^q$
for the Gagliardo-type energy with kernel $\rn{\bx-\by}^{-n-sq}$.
Annular truncations of the critical powers give explicit minimizing sequences with error estimates. Equality is attained only by the zero function in each stated class. For $q=2$, both $L^q$ inequalities reduce to the quadratic-form inequality. Uniform local constancy makes the difference energy finite for every Bruhat--Schwartz function and every $s>0$. The Hardy inequalities on the full test-function space therefore hold precisely when $sq<n$, without the restriction $s<1$.
\end{abstract}


\section{Introduction}\label{sec:intro}
Hardy inequalities bound singularly weighted integrals by differential energies; their sharp constants determine critical coupling strengths for a class of singular potentials. Classical sources include Hardy's work and the monograph of Hardy, Littlewood, and P\'olya~\cite{Hardy,HLP}. In Euclidean space, the classical Hardy inequality reads
\begin{equation*}
 \left(\frac{n-q}{q}\right)^q\int_{\mathbb R^n}\frac{|u(x)|^q}{|x|^q}\dd x
 \le \int_{\mathbb R^n}|\nabla u(x)|^q\dd x,
 \qquad 1<q<n,
\end{equation*}
for $u\in C_c^\infty(\mathbb R^n)$. The corresponding critical homogeneous function lies outside the weighted space and is not an admissible extremizer. We determine the sharp constants and construct minimizing sequences for three fractional Hardy inequalities on $\Qpn$: those involving the Vladimirov--Taibleson quadratic form, the $L^q$ norm of the fractional derivative, and a nonlinear difference energy.

Fractional Hardy inequalities can be formulated in terms of Fourier multipliers, Riesz potentials, or difference energies. The weighted fractional integration theory of Stein and Weiss~\cite{SteinWeiss} derives Hardy inequalities from Riesz potential estimates. Lieb~\cite{Lieb} determined sharp constants in the Hardy--Littlewood--Sobolev inequality and related inequalities. In the quadratic setting, Herbst~\cite{Herbst} studied the critical coupling of a relativistic operator with a Coulomb potential. Beckner~\cite{Beckner} and Yafaev~\cite{Yafaev} investigated sharp constants in Pitt's inequality and the Hardy--Rellich inequalities, respectively, while Frank, Lieb, and Seiringer~\cite{FLS} used fractional Hardy inequalities in spectral estimates for Schr\"odinger operators. In these estimates, the order of the singular potential matches the homogeneity of the operator.

For general exponents $q$, Frank and Seiringer~\cite{FS} developed a nonlinear ground state substitution that expresses the difference between a nonlocal energy and the Hardy term as a nonnegative remainder. Their half-space results~\cite{FShalf}, together with the work of Bogdan and Dyda~\cite{BD}, address the dependence of sharp constants on boundary geometry. Brasco and Franzina~\cite{BF} studied the connection between convexity and Picone-type inequalities. We prove the two-point Picone inequality and its equality case in Section~\ref{sec:nonlinear}.

Different energies lead to different Hardy inequalities in $L^q$. Bogdan, Jakubowski, Lenczewska, and Pietruska-Pa\l uba~\cite{BJLP} considered the energy obtained by pairing the fractional Laplacian with $|u|^{q-2}u$, and its relation to semigroups on $L^q$. For $q\ne2$, this dissipative energy, the functional $\|(-\Delta)^{s/2}u\|_q^q$, and the Gagliardo difference energy are distinct. We consider the $p$-adic counterparts of the latter two.

Taibleson's monograph~\cite{Taibleson} develops Fourier analysis on local fields, including radial multipliers and convolution kernels. Background on $p$-adic analysis, pseudodifferential equations, and distribution theory is provided by Vladimirov, Volovich, and Zelenov~\cite{VVZ}, Kochubei~\cite{KochBook}, and Albeverio, Khrennikov, and Shelkovich~\cite{AKS}, respectively. Haran~\cite{Haran} studied potential theory over the $p$-adics. Rodr\'iguez-Vega and Z\'u\~niga-Galindo~\cite{RZ} used multidimensional Riesz kernels and hypersingular integral representations to study the Vladimirov--Taibleson operator, parabolic equations, and ultrametric diffusion on $\Qpn$. Bendikov, Grigor'yan, Pittet, and Woess~\cite{BGPW} investigated isotropic Markov semigroups on ultrametric spaces and the relations between radial jump kernels, the hierarchy of balls, and diffusion processes.

Kochubei~\cite{Koch2023} established functional inequalities for the Vladimirov--Taibleson operator and studied its Dirichlet problem and boundary H\"older regularity. Bendikov, Grigor'yan, and Molchanov~\cite{BGM} considered hierarchical Schr\"odinger operators with singular potentials. On $\Qp$, for $0<\alpha<1$, their critical potential coefficient is
\[
 -\left[\Gamma_p\left(\frac{1+\alpha}{2}\right)\right]^2.
\]
For $n=1$, this coefficient is the negative of our sharp quadratic-form Hardy constant. We prove the corresponding ground state representation on $\Qpn$ and construct annular minimizing sequences in Section~\ref{sec:L2}; the relations for $q=2$ are given in Section~\ref{sec:compatibility}.

Fu, Wu, and Lu~\cite{FWL} obtained sharp bounds for several $p$-adic Hardy and Hardy--Littlewood--P\'olya averaging operators. Their kernels differ from the Riesz kernel used in the derivative-norm estimate below. Kassymov and Velasquez-Rodriguez~\cite{KV} studied Hardy, Hardy--Littlewood--Sobolev, and Stein--Weiss estimates on constant-order Vilenkin groups and related non-Archimedean groups. Their results include exact norms for some averaging operators, as well as weighted estimates for Riesz potentials and Vladimirov--Taibleson operators. The constants in their Theorem~2.34 and Corollary~2.35 are not specified explicitly. Here we compute the exact norm of $\rn{\bx}^{-s}I_s$ on $L^q(\Qpn)$; the sharp constant is a ratio of $p$-adic Gamma factors.

The nonlinear estimate uses the critical power function and Picone's inequality~\cite{FS,BF}. Points on shells of different radii are separated by the larger radius. The energy of a truncated power function thus reduces to a sum over shell pairs; for pairs inside the annulus, each contribution depends only on the difference of the shell indices. Formulas for the action of fractional operators on power functions are a standard tool in fractional analysis; see, for example, Dyda~\cite{DydaPower} for Euclidean power functions and related eigenvalue problems. The formulas used here follow from $p$-adic Riesz kernels and shell decompositions.

Unlike the Euclidean seminorm defined by first differences~\cite{DNPV}, the difference energy of a Bruhat--Schwartz function has no diagonal singularity: the difference vanishes in a neighborhood of the diagonal. There is therefore no need to assume $s<1$. The condition $sq<n$ instead comes from local integrability of the Hardy weight. The critical case $sq=n$ requires a different weight or function class. Quantitative lower bounds for the subcritical remainder are a further question.

To state the main results, fix a prime $p$ and an integer $n\ge1$, and write
\[
 \rn{\bx}=\max_{1\le j\le n}|x_j|_p,
\]
and normalize Haar measure by $|\Zp^n|=1$. Let $\Sspace$ denote the Bruhat--Schwartz space of compactly supported, locally constant complex-valued functions. For $a>0$, let $D^a$ be the nonnegative Vladimirov--Taibleson operator with symbol $\rn{\bxi}^a$. Our Fourier transform convention is specified in Section~\ref{subsec:Fourier}. Define the $p$-adic Gamma factor by
\begin{equation}\label{eq:Gammadef}
 \GammaN(z)=\frac{1-p^{z-n}}{1-p^{-z}},
 \qquad z\in\mathbb R\setminus\{0\},
\end{equation}
where, for $n=1$, we write $\Gamma_p=\Gamma_p^{(1)}$. We have $\GammaN(z)>0$ for $0<z<n$. Set
\begin{equation}\label{eq:kernelconstant}
 C_{n,a}=-\frac1{\GammaN(-a)}=\frac{p^a-1}{1-p^{-n-a}}>0.
\end{equation}
For $u\in\Sspace$, define
\begin{align}
 \E_\alpha[u]&=\langle D^\alpha u,u\rangle_{L^2},\label{eq:Edef}\\
 \W_{s,q}[u]&=\int_{\Qpn}\rn{\bx}^{-sq}|u(\bx)|^q\dd \bx,\notag\\
 \G_{s,q}[u]&=\iint_{\Qpn\times\Qpn}
 \frac{|u(\bx)-u(\by)|^q}{\rn{\bx-\by}^{n+sq}}\dd \bx\dd \by.\notag
\end{align}

\begin{theorem}[Sharp quadratic-form Hardy inequality]\label{thm:mainL2}
Let $0<\alpha<n$. Then, for every $u\in\Sspace$,
\begin{equation}\label{eq:mainL2}
 \begin{gathered}
 \csharp\int_{\Qpn}\frac{|u(\bx)|^2}{\rn{\bx}^{\alpha}}\dd \bx
 \le \E_\alpha[u],\qquad
 \csharp=\left[\GammaN\left(\frac{n+\alpha}{2}\right)\right]^2
 =\left(\frac{1-p^{-(n-\alpha)/2}}{1-p^{-(n+\alpha)/2}}\right)^2.
 \end{gathered}
\end{equation}
For $h(\bx)=\rn{\bx}^{-(n-\alpha)/2}$, the ground state representation is
\begin{equation}\label{eq:mainGS}
 \E_\alpha[u]-\csharp\int_{\Qpn}\rn{\bx}^{-\alpha}|u(\bx)|^2\dd \bx
 =\frac{C_{n,\alpha}}2\iint
 \frac{h(\bx)h(\by)}{\rn{\bx-\by}^{n+\alpha}}
 \left|\frac{u(\bx)}{h(\bx)}-\frac{u(\by)}{h(\by)}\right|^2\dd \bx\dd \by.
\end{equation}
The constant $\csharp$ is sharp. Both \eqref{eq:mainL2} and \eqref{eq:mainGS} remain valid for measurable functions whose Hardy term and quadratic difference energy are finite, with $\E_\alpha$ defined by the difference integral. In this class, equality holds in \eqref{eq:mainL2} if and only if $u=0$ almost everywhere.
\end{theorem}

\begin{theorem}[Sharp linear $L^q$ Hardy inequality]\label{thm:mainlinear}
Let $1<q<\infty$, $q'=q/(q-1)$, and $0<s<n/q$. Then
\begin{equation}\label{eq:mainlinear}
 \Msharp^q\W_{s,q}[u]\le\|D^su\|_{L^q}^q,
 \qquad u\in\Sspace,
 \qquad
 \Msharp=\frac{\GammaN(n/q'+s)}{\GammaN(n/q')}.
\end{equation}
The constant $\Msharp^q$ is sharp, and equality holds in \eqref{eq:mainlinear} if and only if $u\equiv0$. Let $I_sf=k_s*f$ be the Riesz potential, where $k_s$ is the Riesz kernel defined in Section~\ref{sec:prelim}. The positive integral operator
\[
 T_sf(\bx)=\rn{\bx}^{-s}I_sf(\bx)
\]
is bounded on $L^q(\Qpn)$, with
\begin{equation}\label{eq:mainnorm}
 \|T_s\|_{L^q\to L^q}=\Msharp^{-1}.
\end{equation}
For $f\in L^q(\Qpn)$, the equality $\|T_sf\|_q=\Msharp^{-1}\|f\|_q$ holds if and only if $f=0$ almost everywhere.
\end{theorem}

\begin{theorem}[Sharp Hardy inequality for the nonlinear difference energy]\label{thm:mainnonlinear}
Under the assumptions of Theorem~\ref{thm:mainlinear}, let
\begin{equation}\label{eq:Hseries}
 \Hsharp=2(1-p^{-n})\sum_{m=1}^{\infty}
 p^{-sqm}\bigl(1-p^{-(n/q-s)m}\bigr)^q.
\end{equation}
Then
\begin{equation}\label{eq:mainnonlinear}
 \Hsharp\W_{s,q}[u]\le\G_{s,q}[u],\qquad u\in\Sspace.
\end{equation}
The constant $\Hsharp$ is sharp. Define
\begin{equation}\label{eq:Xdef}
 \X=\{u:\Qpn\to\mathbb C\text{ measurable}:\ \W_{s,q}[u]<\infty,\ \G_{s,q}[u]<\infty\}.
\end{equation}
Inequality \eqref{eq:mainnonlinear} holds for every $u\in\X$, with equality if and only if $u=0$ almost everywhere. The corresponding remainder identity is given in Theorem~\ref{thm:nonlinearidentity}.
\end{theorem}

The results apply to complex-valued, not necessarily radial functions. Local constancy implies that a test function vanishing almost everywhere vanishes identically. Each inequality is therefore strict for nonzero $u\in\Sspace$.

Sharpness of the two $L^q$ inequalities is established using the radial sequence
\begin{equation}\label{eq:criticalseqintro}
 \begin{gathered}
 w(\bx)=\rn{\bx}^{-(n/q-s)},\qquad u_N=w\ind_{A_N},\\
 A_N=\{\bx:p^{-N}\le\rn{\bx}\le p^N\}=B_{-N}\setminus B_{N+1}.
 \end{gathered}
\end{equation}
Here $N$ is a positive integer and $B_k=p^k\Zp^n$. Since $sq+q(n/q-s)=n$, each shell contributes equally to $\W_{s,q}[u_N]$, so the Hardy term grows linearly with the number of shells. The linear truncation error is uniformly bounded in $L^q$; shell summation controls the boundary contributions to the quadratic and nonlinear energies. Dividing by the Hardy term gives convergence to the sharp constants.

The proofs use, respectively, a quadratic ground state representation, a pair of adjoint Schur identities, and Picone's inequality applied to the critical power function. The equality cases follow from the ground state representation, H\"older's inequality, and Picone's inequality, respectively. For $q=2$, Plancherel's identity and the jump-kernel formula relate the three results.

Section~\ref{sec:prelim} collects the shell and Riesz kernel formulas, the integral representation of the Vladimirov--Taibleson operator, Riesz inversion, and the necessity of the parameter range. Section~\ref{sec:L2} proves Theorem~\ref{thm:mainL2} and determines the critical coupling of the associated singular Schr\"odinger form. Sections~\ref{sec:schur} and \ref{sec:nonlinear} prove Theorems~\ref{thm:mainlinear} and \ref{thm:mainnonlinear}, respectively. The final subsection gives the relations for $q=2$.

\section{Preliminaries}\label{sec:prelim}
\subsection{The \texorpdfstring{$p$}{p}-adic field and the hierarchy of balls}\label{subsec:setting}
Fix a prime $p$. For a nonzero rational number $t=p^v a/b$, where $v\in\mathbb Z$ and $p\nmid ab$, set $|t|_p=p^{-v}$, and put $|0|_p=0$. The completion of $\mathbb Q$ with respect to this absolute value is denoted by $\Qp$. The $p$-adic valuation is $v_p(t)=-\log_p|t|_p$, with $v_p(0)=+\infty$. Every nonzero $t\in\Qp$ has a unique expansion
\[
 t=\sum_{j=v_p(t)}^\infty a_jp^j,
 \qquad a_j\in\{0,\ldots,p-1\},\quad a_{v_p(t)}\ne0.
\]
Write $\Zp=\{t\in\Qp:|t|_p\le1\}$, whose group of units is $\Zp^*=\{t\in\Qp:|t|_p=1\}$. For background on $p$-adic analysis, see \cite{Taibleson,VVZ,AKS}.

For $\bx=(x_1,\ldots,x_n)\in\Qpn$, define
\[
 \rn{\bx}=\max_{1\le i\le n}|x_i|_p,
 \qquad v_p(\bx)=\min_{1\le i\le n}v_p(x_i).
\]
Then $\rn{\bx}=p^{-v_p(\bx)}$ and $\rn{\bx+\by}\le\max\{\rn{\bx},\rn{\by}\}$. We use product Haar measure $\dd \bx=\dd x_1\cdots\dd x_n$, normalized by $|\Zp^n|=1$. For a measurable set $A$, its Haar measure is denoted by $|A|$ and its indicator by $\ind_A$. For $k\in\mathbb Z$ and $\boldsymbol a\in\Qpn$, set
\begin{equation*}
 B_k(\boldsymbol a)=\boldsymbol a+p^k\Zp^n
 =\{\bx\in\Qpn:\rn{\bx-\boldsymbol a}\le p^{-k}\},
 \qquad B_k=B_k(\boldsymbol0).
\end{equation*}
Thus $B_{k+1}\subset B_k$, with radius tending to zero as $k\to+\infty$. Set
\begin{equation*}
 S_k=B_k\setminus B_{k+1}
 =\{\bx\in\Qpn:v_p(\bx)=k\}
 =\{\bx\in\Qpn:\rn{\bx}=p^{-k}\}.
\end{equation*}
The index of $S_k$ is the common valuation of its points; thus $S_k$ lies outside $S_\ell$ when $k<\ell$.

\subsection{Function spaces and the Fourier transform}\label{subsec:Fourier}
\begin{definition}[Bruhat--Schwartz space]\label{def:Schwartz}
A function $u$ is locally constant if every $\bx\in\Qpn$ has a neighborhood $B_\ell(\bx)$ on which $u$ is constant. We use the test-function space
\[
 \Sspace=\{u:\Qpn\to\mathbb C:\ u\text{ is locally constant and compactly supported}\},
\]
equipped with its usual test-function topology. Its continuous dual $\Sp$ is the space of distributions.
\end{definition}

For $t\in\Qp$, the $p$-adic fractional part and the standard additive character are defined by
\[
 \{t\}_p=
 \begin{cases}
 0,&v_p(t)\ge0,\\
 \displaystyle\sum_{j=v_p(t)}^{-1}a_jp^j,&v_p(t)<0,
 \end{cases}
 \qquad \chi_p(t)=e^{2\pi i\{t\}_p}.
\]
In particular, the kernel of $\chi_p$ is $\Zp$. Write $\bx\cdot\bxi=\sum_{i=1}^n x_i\xi_i$ and set $\chi^p_{\bxi}(\bx)=\chi_p(-\bx\cdot\bxi)$. For $u\in L^1(\Qpn)$, define
\begin{equation*}
 \hat u(\bxi)
 =\int_{\Qpn}u(\bx)\overline{\chi^p_{\bxi}(\bx)}\dd \bx
 =\int_{\Qpn}u(\bx)\chi_p(\bx\cdot\bxi)\dd \bx.
\end{equation*}
For $u\in\Sspace$, Fourier inversion takes the form
\begin{equation*}
 u(\bx)=\int_{\Qpn}\hat u(\bxi)\chi_p(-\bx\cdot\bxi)\dd\bxi.
\end{equation*}
The Fourier transform maps $\Sspace$ onto itself and extends to a unitary operator on $L^2(\Qpn)$. We abbreviate $\|f\|_r=\|f\|_{L^r(\Qpn)}$ and use the inner product
\[
 \langle f,g\rangle_{L^2}
 =\int_{\Qpn}f(\bx)\overline{g(\bx)}\dd \bx
 =\int_{\Qpn}\hat f(\bxi)\overline{\hat g(\bxi)}\dd\bxi.
\]
Unless otherwise stated, single integrals are over $\Qpn$ and double integrals are over $\Qpn\times\Qpn$. In estimates, $C>0$ denotes a constant depending only on the fixed parameters; its value may change from line to line. Constants in the truncation estimates are independent of $N$. Values at the origin or on the diagonal do not affect the integrals. Identities involving power functions or kernels hold almost everywhere unless stated pointwise.

\begin{definition}[Vladimirov--Taibleson operator]\label{def:Taibleson}
For $a>0$, the operator $D^a$ is the nonnegative self-adjoint Fourier multiplier on $L^2(\Qpn)$ given by
\begin{equation*}
 \widehat{D^au}(\bxi)=\rn{\bxi}^a\hat u(\bxi),
 \qquad
 \mathcal D(D^a)=\{u\in L^2(\Qpn):\rn{\bxi}^a\hat u(\bxi)\in L^2(\Qpn)\}.
\end{equation*}
\end{definition}

For $t\ge0$, define the Sobolev space
\[
 H^t(\Qpn)=\left\{u\in L^2(\Qpn):
 \int_{\Qpn}\max\{1,\rn{\bxi}\}^{2t}|\hat u(\bxi)|^2\dd\bxi<\infty\right\}.
\]
Then $\mathcal D(D^a)=H^a(\Qpn)$, and the quadratic-form domain is $H^{a/2}(\Qpn)$. For functions outside $L^2$, we define the energy by the corresponding difference integral.

\begin{lemma}[Shell integration and the Fourier transform of a ball]\label{lem:shell}
The measures of balls and shells are $|B_k|=p^{-nk}$ and $|S_k|=(1-p^{-n})p^{-nk}$, respectively. If $\rn{\bx}\ne\rn{\by}$, then
\begin{equation}\label{eq:ultradistance}
 \rn{\bx-\by}=\max\{\rn{\bx},\rn{\by}\}.
\end{equation}
In particular, if $\bx\in S_k$, $\by\in S_\ell$, and $k\ne\ell$, then $\rn{\bx-\by}=p^{-\min\{k,\ell\}}$. For every integrable or nonnegative radial function $\bx\mapsto F(\rn{\bx})$,
\begin{equation}\label{eq:shellsum}
 \int_{\Qpn}F(\rn{\bx})\dd \bx=(1-p^{-n})\sum_{k\in\mathbb Z}p^{-nk}F(p^{-k}).
\end{equation}
Moreover,
\begin{equation}\label{eq:ballFourier}
 \widehat{\ind_{B_k}}(\bxi)=p^{-nk}\ind_{B_{-k}}(\bxi).
\end{equation}
\end{lemma}
\begin{proof}
The measure formulas follow from Haar scaling, $\dd(a\bx)=|a|_p^n\dd \bx$. Suppose without loss of generality that $\rn{\bx}<\rn{\by}$. The ultrametric inequality gives $\rn{\bx-\by}\le\rn{\by}$. A strict inequality would contradict the ultrametric inequality applied to $\by=\bx-(\bx-\by)$, so \eqref{eq:ultradistance} follows. Decomposition into shells gives \eqref{eq:shellsum}.

If $\bxi\in B_{-k}$, then $\bx\cdot\bxi\in\Zp$ for every $\bx\in B_k$. The character is therefore identically $1$ on $B_k$, and its integral equals $p^{-nk}$. If $\bxi\notin B_{-k}$, choose $\bz\in B_k$ such that $\chi_p(\bz\cdot\bxi)\ne1$. Translation invariance shows that the integral equals itself multiplied by $\chi_p(\bz\cdot\bxi)$; hence it vanishes, proving \eqref{eq:ballFourier}.
\end{proof}

\begin{lemma}[Uniform local constancy of test functions]\label{lem:localconstant}
For every $u\in\Sspace$, there is an integer $\ell$ such that
\begin{equation}\label{eq:uniformlocal}
 u(\bx+\bz)=u(\bx)\qquad(\bx\in\Qpn,\ \bz\in B_\ell).
\end{equation}
Consequently, for all $a>0$ and $q\ge1$,
\[
 \iint\frac{|u(\bx)-u(\by)|^q}{\rn{\bx-\by}^{n+a}}\dd \bx\dd \by<\infty.
\]
If $0<a<n$, then $\int\rn{\bx}^{-a}|u(\bx)|^q\dd \bx<\infty$.
\end{lemma}
\begin{proof}
Compactness of $\supp u$ gives a finite cover by balls on which $u$ is constant. Let $p^{-\ell}$ be the smallest radius among these balls. Then \eqref{eq:uniformlocal} holds for $\bx\in\supp u$. If $\bx\notin\supp u$ but $\bx+\bz\in\supp u$, applying this conclusion at $\bx+\bz$ would give $u(\bx)=u(\bx+\bz)\ne0$, a contradiction. Thus \eqref{eq:uniformlocal} holds on all of $\Qpn$.

Set $\by=\bx-\bz$. The difference vanishes for $\bz\in B_\ell$, while
\[
 \int|u(\bx)-u(\bx-\bz)|^q\dd \bx\le2^q\|u\|_q^q.
\]
Together with
\[
 \int_{\Qpn\setminus B_\ell}\rn{\bz}^{-n-a}\dd\bz
 =(1-p^{-n})\sum_{k<\ell}p^{ak}<\infty
\]
this proves finiteness of the difference integral. Since $u$ is bounded and compactly supported, finiteness of the Hardy integral follows from $(1-p^{-n})\sum_{k\ge0}p^{-(n-a)k}<\infty$ for $a<n$.
\end{proof}

\subsection{The Fourier transform and semigroup property of Riesz kernels}
The Gamma factor \eqref{eq:Gammadef} satisfies
\begin{equation}\label{eq:reflection}
 \GammaN(z)\GammaN(n-z)=1,\qquad 0<z<n.
\end{equation}
For $0<z<n$, set
\begin{equation*}
 k_z(\bx)=\frac{\rn{\bx}^{z-n}}{\GammaN(z)},\qquad \bx\ne\boldsymbol0.
\end{equation*}
The function $k_z$ is positive and locally integrable, and hence defines a distribution in $\Sp$.

\begin{lemma}[Fourier transform of the Riesz kernel]\label{lem:RieszFourier}
In $\Sp$,
\begin{equation*}
 \widehat{k_z}(\bxi)=\rn{\bxi}^{-z},\qquad 0<z<n.
\end{equation*}
\end{lemma}
\begin{proof}
If $\rn{\bx}=p^{-j}$, then $\ind_{B_k}(\bx)=1$ precisely when $k\le j$. Summing a geometric series gives
\[
 \rn{\bx}^{z-n}=(1-p^{z-n})\sum_{k\in\mathbb Z}p^{(n-z)k}\ind_{B_k}(\bx).
\]
After pairing with a test function, the terms as $k\to+\infty$ and $k\to-\infty$ are bounded by constant multiples of $p^{-zk}$ and $p^{(n-z)k}$, respectively. Termwise integration is therefore justified. By \eqref{eq:ballFourier},
\[
 \F(\rn{\bx}^{z-n})
 =(1-p^{z-n})\sum_{k\in\mathbb Z}p^{-zk}\ind_{B_{-k}}(\bxi).
\]
For $\rn{\bxi}=p^{-j}$, only the terms with $k\ge-j$ remain, and the sum equals
\[
 \frac{1-p^{z-n}}{1-p^{-z}}\,p^{zj}
 =\GammaN(z)\rn{\bxi}^{-z}.
\]
The right-hand side is locally integrable because $z<n$. The transformed series also converges absolutely when paired with a test function, so the identity holds in $\Sp$.
\end{proof}

\begin{lemma}[Riesz semigroup identity]\label{lem:semigroup}
Let $a,b>0$ and $a+b<n$. For every $\bx\ne\boldsymbol0$, the convolution $k_a*k_b$ converges absolutely and satisfies
\begin{equation}\label{eq:semigroup}
 (k_a*k_b)(\bx)=k_{a+b}(\bx).
\end{equation}
\end{lemma}
\begin{proof}
First suppose $\rn{\bx}=1$. Decompose the integral
\[
 \int_{\Qpn}\rn{\bx-\by}^{a-n}\rn{\by}^{b-n}\dd \by
\]
into the regions $\rn{\by}<1$, $\rn{\bx-\by}<1$, $\rn{\by}>1$, and $\rn{\by}=\rn{\bx-\by}=1$. The first two regions are the disjoint balls $B_1$ and $\bx+B_1$, and the fourth has measure $1-2p^{-n}$. By \eqref{eq:ultradistance} and shell decomposition,
\begin{align*}
 \int_{\Qpn}\rn{\bx-\by}^{a-n}\rn{\by}^{b-n}\dd \by&=(1-p^{-n})\left(\sum_{m\ge1}p^{-bm}+\sum_{m\ge1}p^{-am}
   +\sum_{m\ge1}p^{-(n-a-b)m}\right)+1-2p^{-n}\\
 &=\frac{1-p^{-n}}{p^b-1}+\frac{1-p^{-n}}{p^a-1}
   +\frac{1-p^{-n}}{p^{n-a-b}-1}+1-2p^{-n}\\
 &=\frac{(1-p^{a-n})(1-p^{b-n})(1-p^{-a-b})}
 {(1-p^{-a})(1-p^{-b})(1-p^{a+b-n})}\\
 &=\frac{\GammaN(a)\GammaN(b)}{\GammaN(a+b)}.
\end{align*}
The last equality uses \eqref{eq:Gammadef}. The three series converge because $b>0$, $a>0$, and $a+b<n$, respectively.

For a general point with $\rn{\bx}=p^{-j}$, write $\bx=p^j\bx_0$ and $\by=p^j\bz$, where $\rn{\bx_0}=1$. Scaling contributes $p^{-(a+b-n)j}$, and the calculation applies to every $\rn{\bx_0}=1$. Hence
\[
 \int_{\Qpn}\rn{\bx-\by}^{a-n}\rn{\by}^{b-n}\dd \by
 =\frac{\GammaN(a)\GammaN(b)}{\GammaN(a+b)}\rn{\bx}^{a+b-n}.
\]
Dividing by $\GammaN(a)\GammaN(b)$ proves \eqref{eq:semigroup}.
\end{proof}

\subsection{Hypersingular integrals and power functions}
\begin{proposition}[Integral representation of the Vladimirov--Taibleson operator]\label{prop:integral}
For $a>0$ and $u\in\Sspace$,
\begin{equation}\label{eq:Taiblesonintegral}
 \begin{aligned}
 D^au(\bx)
 &=\frac{1-p^a}{1-p^{-a-n}}
   \int_{\Qpn}\frac{u(\by)-u(\bx)}{\rn{\by-\bx}^{n+a}}\dd \by\\
 &=C_{n,a}\int_{\Qpn}\frac{u(\bx)-u(\by)}{\rn{\bx-\by}^{n+a}}\dd \by.
 \end{aligned}
\end{equation}
The integrals converge absolutely at every $\bx$. Moreover, $D^au\in L^1\cap L^\infty$, and hence $D^au\in L^r$ for $1\le r\le\infty$. If $\supp u\subset B_K$, then
\begin{equation}\label{eq:Dstail}
 D^au(\bx)=-C_{n,a}\rn{\bx}^{-n-a}\int_{\Qpn}u(\by)\dd \by,
 \qquad \rn{\bx}>p^{-K}.
\end{equation}
In addition, $\int D^au\dd \bx=0$.
\end{proposition}
\begin{proof}
By Lemma~\ref{lem:localconstant}, there is an integer $\ell$ such that the numerator vanishes whenever $\rn{\bx-\by}\le p^{-\ell}$. The kernel is integrable on the remaining region, so the integral defines a bounded function. Furthermore,
\[
 \int_{\Qpn}\!\int_{\Qpn}
 \frac{|u(\bx)-u(\bx-\bz)|}{\rn{\bz}^{n+a}}\dd\bz\dd \bx
 \le2\|u\|_1\int_{\Qpn\setminus B_\ell}\rn{\bz}^{-n-a}\dd\bz<\infty,
\]
so the function belongs to $L^1$ and Fubini's theorem applies when computing its Fourier transform.

For $\rn{\bxi}=1$, formula \eqref{eq:ballFourier} gives
\[
 \int\frac{1-\chi_p(\bz\cdot\bxi)}{\rn{\bz}^{n+a}}\dd\bz
 =p^{-a}+(1-p^{-n})\sum_{m=2}^{\infty}p^{-am}
 =\frac{1-p^{-n-a}}{p^a-1}.
\]
Indeed, the integral of the character over $S_k$ is $-1$ when $k=-1$ and zero when $k\le-2$, while the character is identically $1$ when $k\ge0$. The series above is indexed by $m=-k\ge2$. For general $\bxi\ne\boldsymbol0$, scaling contributes the factor $\rn{\bxi}^a$; at $\bxi=\boldsymbol0$, the integral vanishes. Multiplication by $C_{n,a}\hat u(\bxi)$ therefore gives the Fourier multiplier $\rn{\bxi}^a\hat u(\bxi)$, proving \eqref{eq:Taiblesonintegral}. Evaluating at $\bxi=\boldsymbol0$ also gives $\int D^au\dd \bx=0$.

If $\rn{\bx}>p^{-K}$, then $u(\bx)=0$ and $\rn{\bx-\by}=\rn{\bx}$ for every $\by\in\supp u$. Substitution in the integral representation gives \eqref{eq:Dstail}.
\end{proof}

\begin{corollary}[Jump-kernel formula for the quadratic form]\label{cor:quadratic}
For $a>0$ and $u\in\Sspace$,
\begin{equation}\label{eq:quadratic}
 \E_a[u]=\int_{\Qpn}\rn{\bxi}^a|\hat u(\bxi)|^2\dd\bxi
 =\frac{C_{n,a}}2\iint\frac{|u(\bx)-u(\by)|^2}{\rn{\bx-\by}^{n+a}}\dd \bx\dd \by.
\end{equation}
\end{corollary}
\begin{proof}
The first equality is Plancherel's identity. For the second, multiply \eqref{eq:Taiblesonintegral} by $\overline{u(\bx)}$, integrate, interchange $\bx$ and $\by$, and average the two expressions. The required absolute integrability follows from Lemma~\ref{lem:localconstant} and the proof of Proposition~\ref{prop:integral}.
\end{proof}

\begin{lemma}[Action on power functions]\label{lem:power}
Let $a>0$ and $0<\delta<n-a$. Substituting $u(\bx)=\rn{\bx}^{-\delta}$ into \eqref{eq:Taiblesonintegral} gives an absolutely convergent integral at every $\bx\ne\boldsymbol0$, and
\begin{equation}\label{eq:power}
 D^a\bigl(\rn{\cdot}^{-\delta}\bigr)(\bx)
 =\frac{\GammaN(n-\delta)}{\GammaN(n-\delta-a)}
 \rn{\bx}^{-\delta-a}.
\end{equation}
Equivalently, if $0<\eta-a<\eta<n$, then
\[
 D^a\rn{\bx}^{\eta-n}
 =\frac{\GammaN(\eta)}{\GammaN(\eta-a)}\rn{\bx}^{\eta-a-n}.
\]
\end{lemma}
\begin{proof}
First let $\rn{\bx}=1$. If $\by$ lies on the same shell as $\bx$, then $\rn{\by}^{-\delta}=\rn{\bx}^{-\delta}=1$, so the integrand vanishes. The outer and inner shells contribute, respectively,
\[
 (1-p^{-n})\sum_{m\ge1}p^{-am}(1-p^{-\delta m}),\qquad
 (1-p^{-n})\sum_{m\ge1}p^{-nm}(1-p^{\delta m}).
\]
Both series converge absolutely because $a>0$ and $\delta<n$. Adding them and multiplying by $C_{n,a}$ gives
\begin{align*}
 D^a\bigl(\rn{\cdot}^{-\delta}\bigr)(\bx)
 &=C_{n,a}(1-p^{-n})\left(
 \frac1{p^a-1}-\frac1{p^{a+\delta}-1}
 -\frac1{p^{n-\delta}-1}+\frac1{p^n-1}\right)\\
 &=\frac{(1-p^{-\delta})(1-p^{-(n-\delta-a)})}
 {(1-p^{-(n-\delta)})(1-p^{-(a+\delta)})}
 =\frac{\GammaN(n-\delta)}{\GammaN(n-\delta-a)}.
\end{align*}
The last two equalities follow from \eqref{eq:kernelconstant} and \eqref{eq:Gammadef}. For general $\bx\ne\boldsymbol0$, write $\bx=p^j\bx_0$, where $\rn{\bx_0}=1$, and set $\by=p^j\bz$. The difference, kernel, and measure contribute the factors $p^{\delta j}$, $p^{(n+a)j}$, and $p^{-nj}$, respectively. Their product is $p^{(\delta+a)j}=\rn{\bx}^{-\delta-a}$, which proves \eqref{eq:power}.

One can also derive \eqref{eq:power} by Fourier transformation. By Lemma~\ref{lem:RieszFourier}, the Fourier transform of $\rn{\cdot}^{-\delta}$ is $\GammaN(n-\delta)\rn{\bxi}^{\delta-n}$. Multiply by $\rn{\bxi}^a$ and apply the same lemma with parameter $n-\delta-a\in(0,n)$. All the power kernels involved are locally integrable, so the formula also holds in the sense of distributions.
\end{proof}

\subsection{Riesz inversion and the parameter range}
\begin{lemma}[Riesz inversion for test functions]\label{lem:inverse}
Let $0<s<n$, $u\in\Sspace$, and $f=D^su$. Then
\begin{equation*}
 I_sf(\bx):=(k_s*f)(\bx)
\end{equation*}
converges absolutely at every $\bx$, and $I_sD^su=u$ almost everywhere.
\end{lemma}
\begin{proof}
Proposition~\ref{prop:integral} gives $f\in L^1\cap L^\infty$. Since $k_s$ is integrable near the origin and bounded for $\rn{\bz}>1$,
\[
 \int k_s(\bx-\by)|f(\by)|\dd \by
 \le\|f\|_\infty\int_{B_0}k_s(\bz)\dd\bz
     +\frac{\|f\|_1}{\GammaN(s)}<\infty.
\]
The bound is independent of $\bx$, so $I_sf$ is bounded.

For $\varphi\in\Sspace$, Lemma~\ref{lem:RieszFourier} gives
\[
 I_s\varphi=\F^{-1}\big(\rn{\bxi}^{-s}\hat\varphi(\bxi)\big).
\]
Here $\rn{\bxi}^{-s}\hat\varphi(\bxi)$ has compact support and belongs to $L^1$ because $s<n$. Also, $I_s|\varphi|$ is bounded, and hence
\[
 \int\!\int |f(\by)|k_s(\bx-\by)|\varphi(\bx)|\dd \bx\dd \by<\infty.
\]
Symmetry of the kernel and Fourier inversion now give
\begin{align*}
 \int I_sf(\bx)\varphi(\bx)\dd \bx
 &=\int f(\by)I_s\varphi(\by)\dd \by\\
 &=\int\hat f(-\bxi)\rn{\bxi}^{-s}\hat\varphi(\bxi)\dd\bxi\\
 &=\int\hat u(-\bxi)\hat\varphi(\bxi)\dd\bxi\\
 &=\int u(\bx)\varphi(\bx)\dd \bx.
\end{align*}
The second step follows by Fubini's theorem, since $f\in L^1$ and $\rn{\bxi}^{-s}\hat\varphi\in L^1$. Thus $I_sf=u$ in $\Sp$ and, by local integrability, almost everywhere.
\end{proof}

\begin{remark}\label{rem:inverse}
By \eqref{eq:Dstail}, $D^su$ is not compactly supported if $\int u\ne0$, so $D^s$ need not preserve $\Sspace$. The proof of Lemma~\ref{lem:inverse} uses $D^su\in L^1\cap L^\infty$ without requiring $D^su$ to be a test function.
\end{remark}

For the $L^q$ Hardy inequalities we assume
\begin{equation}\label{eq:qrange}
 1<q<\infty,\qquad q'=\frac{q}{q-1},\qquad 0<s<\frac nq.
\end{equation}
The condition $sq<n$ is necessary on the full test-function space.

\begin{proposition}[Necessity of the parameter range]\label{prop:necessaryrange}
Fix $q>1$ and $s>0$. If $sq\ge n$, no positive constant $C$ makes either of the following inequalities valid for all $u\in\Sspace$:
\[
 C\W_{s,q}[u]\le\|D^su\|_q^q,
 \qquad
 C\W_{s,q}[u]\le\G_{s,q}[u].
\]
\end{proposition}
\begin{proof}
Take $u=\ind_{B_0}$. This is a test function, but
\[
 \W_{s,q}[u]=(1-p^{-n})\sum_{k\ge0}p^{-(n-sq)k}=\infty
 \qquad(sq\ge n).
\]
On the other hand, Proposition~\ref{prop:integral} gives $D^su\in L^q$. More explicitly,
\begin{equation*}
 D^s\ind_{B_0}(\bx)=
 \begin{cases}
 \displaystyle\frac{1-p^{-n}}{1-p^{-n-s}},&\rn{\bx}\le1,\\[4pt]
 -C_{n,s}\rn{\bx}^{-n-s},&\rn{\bx}>1.
 \end{cases}
\end{equation*}
The difference is nonzero only when one point lies in $B_0$ and the other in $B_0^c$, so
\[
 \G_{s,q}[\ind_{B_0}]
 =2\int_{B_0}\int_{B_0^c}\rn{\bx-\by}^{-n-sq}\dd \by\dd \bx
 =2(1-p^{-n})\sum_{k\le-1}p^{sqk}
 =\frac{2(1-p^{-n})}{p^{sq}-1}<\infty.
\]
Both right-hand sides are finite, while the left-hand side is infinite for $C>0$.
\end{proof}

\begin{remark}[Local constancy and the order of the operator]\label{rem:slarge}
In Euclidean space, first differences of a smooth function generally vanish to first order at the diagonal, leading to the usual range $0<s<1$ for the Gagliardo seminorm~\cite{DNPV}. For $u\in\Sspace$, the difference $u(\bx)-u(\by)$ vanishes below a fixed radius by Lemma~\ref{lem:localconstant}. The energy is therefore finite for all $s>0$, whereas $sq<n$ is needed for integrability of the Hardy weight at the origin.

Theorems~\ref{thm:mainlinear} and \ref{thm:mainnonlinear} and Proposition~\ref{prop:necessaryrange} give the exact range $0<s<n/q$ on the full test-function space. This does not assert finite difference energy for arbitrary measurable functions. For functions vanishing near the origin, the case $sq\ge n$ is a separate question.
\end{remark}

\section{The sharp quadratic-form Hardy inequality}\label{sec:L2}
We prove Theorem~\ref{thm:mainL2}. Fix $0<\alpha<n$ and consider the critical power function
\begin{equation*}
 h(\bx)=\rn{\bx}^{-(n-\alpha)/2}.
\end{equation*}
For a measurable function $u:\Qpn\to\mathbb C$, set
\[
 \E_\alpha[u]=\frac{C_{n,\alpha}}2
 \iint\frac{|u(\bx)-u(\by)|^2}{\rn{\bx-\by}^{n+\alpha}}\dd \bx\dd \by\in[0,\infty],
\]
and let
\begin{equation*}
 \mathcal Y_\alpha=
 \left\{u:\Qpn\to\mathbb C\text{ measurable}:\ \E_\alpha[u]<\infty,\quad
 \int\rn{\bx}^{-\alpha}|u(\bx)|^2\dd \bx<\infty\right\}.
\end{equation*}
By \eqref{eq:quadratic}, this agrees with \eqref{eq:Edef} on $\Sspace$; outside $L^2$, $\E_\alpha$ is understood as the difference integral.

\subsection{The critical power function and the ground state representation}
\begin{lemma}[Equation for the critical power function]\label{lem:L2ground}
For $\bx\ne\boldsymbol0$,
\begin{equation}\label{eq:L2ground}
 D^\alpha h(\bx)=\csharp\rn{\bx}^{-\alpha}h(\bx).
\end{equation}
Moreover,
\begin{equation}\label{eq:L2abs}
 \int\frac{|h(\bx)-h(\by)|}{\rn{\bx-\by}^{n+\alpha}}\dd \by
 =(1-p^{-n})\left(\frac1{p^\alpha-1}-\frac1{p^n-1}\right)
 \rn{\bx}^{-\alpha}h(\bx)<\infty.
\end{equation}
\end{lemma}
\begin{proof}
Apply Lemma~\ref{lem:power} with $a=\alpha$ and $\delta=(n-\alpha)/2$. By the reflection identity \eqref{eq:reflection},
\[
 \frac{\GammaN((n+\alpha)/2)}{\GammaN((n-\alpha)/2)}
 =\left[\GammaN\left(\frac{n+\alpha}{2}\right)\right]^2=\csharp,
\]
which gives \eqref{eq:L2ground}.

To compute the absolute integral, first suppose $\rn{\bx}=1$. The outer shell $S_{-m}$ contributes
\[
 (1-p^{-n})p^{-\alpha m}\bigl(1-p^{-(n-\alpha)m/2}\bigr),
\]
and the inner shell $S_m$ contributes
\[
 (1-p^{-n})p^{-(n+\alpha)m/2}\bigl(1-p^{-(n-\alpha)m/2}\bigr).
\]
Since $(n+\alpha)/2=\alpha+(n-\alpha)/2$, their sum is
\begin{align*}
 (1-p^{-n})p^{-\alpha m}
   \bigl(1-p^{-(n-\alpha)m/2}\bigr)\bigl(1+p^{-(n-\alpha)m/2}\bigr)=(1-p^{-n})(p^{-\alpha m}-p^{-nm}).
\end{align*}
Summation over $m\ge1$, followed by scaling, proves \eqref{eq:L2abs}.
\end{proof}

\begin{proposition}[Ground state representation]\label{prop:L2identity}
For $u\in\mathcal Y_\alpha$, set $v=u/h$. Then
\begin{equation}\label{eq:L2identity}
 \E_\alpha[u]-\csharp\int\rn{\bx}^{-\alpha}|u(\bx)|^2\dd \bx
 =\frac{C_{n,\alpha}}2\iint
 \frac{h(\bx)h(\by)|v(\bx)-v(\by)|^2}{\rn{\bx-\by}^{n+\alpha}}\dd \bx\dd \by.
\end{equation}
The right-hand side is finite and nonnegative.
\end{proposition}
\begin{proof}
For $A,B>0$ and $z,t\in\mathbb C$, expansion gives
\begin{equation}\label{eq:L2algebra}
 |Az-Bt|^2-AB|z-t|^2
 =(A-B)(A|z|^2-B|t|^2).
\end{equation}
Take $A=h(\bx)$, $B=h(\by)$, $z=v(\bx)$, and $t=v(\by)$. Lemma~\ref{lem:L2ground} gives the integrability estimate
\begin{align*}
 &\iint\frac{|h(\bx)-h(\by)|}{\rn{\bx-\by}^{n+\alpha}}
 \left(\frac{|u(\bx)|^2}{h(\bx)}+\frac{|u(\by)|^2}{h(\by)}\right)\dd \bx\dd \by\\
 =&2\int\frac{|u(\bx)|^2}{h(\bx)}
 \left(\int\frac{|h(\bx)-h(\by)|}{\rn{\bx-\by}^{n+\alpha}}\dd \by\right)\dd \bx\\
 =&2(1-p^{-n})\left(\frac1{p^\alpha-1}-\frac1{p^n-1}\right)
 \int\rn{\bx}^{-\alpha}|u(\bx)|^2\dd \bx\\
 <&\infty.
\end{align*}
Thus the right-hand side of \eqref{eq:L2algebra}, multiplied by the kernel, is absolutely integrable. Fubini's theorem gives
\begin{align*}
 &\frac{C_{n,\alpha}}2\iint
 \frac{(h(\bx)-h(\by))\big(|u(\bx)|^2/h(\bx)-|u(\by)|^2/h(\by)\big)}{\rn{\bx-\by}^{n+\alpha}}\dd \bx\dd \by\\
 =&C_{n,\alpha}\int\frac{|u(\bx)|^2}{h(\bx)}
 \left(\int\frac{h(\bx)-h(\by)}{\rn{\bx-\by}^{n+\alpha}}\dd \by\right)\dd \bx\\
 =&\int\frac{|u(\bx)|^2}{h(\bx)}D^\alpha h(\bx)\dd \bx\\
 =&\csharp\int\rn{\bx}^{-\alpha}|u(\bx)|^2\dd \bx.
\end{align*}
Integration of \eqref{eq:L2algebra} now gives \eqref{eq:L2identity}. The remainder is finite because $\E_\alpha[u]<\infty$.
\end{proof}

\begin{corollary}[Hardy inequality and equality case]\label{cor:L2Hardy}
Inequality \eqref{eq:mainL2} holds for every $u\in\mathcal Y_\alpha$, with equality if and only if $u=0$ almost everywhere. In particular, equality in $\Sspace$ holds only for $u\equiv0$.
\end{corollary}
\begin{proof}
The remainder in \eqref{eq:L2identity} gives the inequality and vanishes in the equality case. Since $h(\bx)h(\by)\rn{\bx-\by}^{-n-\alpha}>0$ for almost every $(\bx,\by)$, we have $v(\bx)=v(\by)$ for almost every pair. Fubini's theorem implies that $v=c$ almost everywhere for some $c\in\mathbb C$. If $c\ne0$, then
\[
 \int\rn{\bx}^{-\alpha}|u(\bx)|^2\dd \bx
 =|c|^2\int\rn{\bx}^{-n}\dd \bx
 =|c|^2(1-p^{-n})\sum_{k\in\mathbb Z}1=\infty,
\]
contrary to $u\in\mathcal Y_\alpha$. Hence $c=0$ and $u=0$ almost everywhere. Conversely, the zero function attains equality.
\end{proof}

\subsection{Truncation and sharpness}
The next lemma gives convergence of the quadratic energy and Hardy term under truncation near the origin. This also allows the Hardy inequality to pass from test functions supported away from the origin to all of $\Sspace$.

\begin{lemma}[Truncation near the origin]\label{lem:origin}
For every integer $k$,
\begin{align}
 \E_\alpha[\ind_{B_k}]&=
 \frac{1-p^{-n}}{1-p^{-(n+\alpha)}}p^{-(n-\alpha)k},\label{eq:ballenergy}\\
 \int\rn{\bx}^{-\alpha}\ind_{B_k}(\bx)\dd \bx
 &=\frac{1-p^{-n}}{1-p^{-(n-\alpha)}}p^{-(n-\alpha)k}.\label{eq:ballweight}
\end{align}
Consequently, if $u\in\Sspace$ and $u_M=u\ind_{\Qpn\setminus B_M}$, then, as $M\to\infty$,
\[
 \E_\alpha[u-u_M]\longrightarrow0,
 \qquad
 \int\rn{\bx}^{-\alpha}|u-u_M|^2\dd \bx\longrightarrow0.
\]
\end{lemma}
\begin{proof}
By \eqref{eq:ballFourier} and \eqref{eq:quadratic},
\[
 \E_\alpha[\ind_{B_k}]
 =p^{-2nk}\int_{B_{-k}}\rn{\bxi}^{\alpha}\dd\bxi
 =p^{-2nk}(1-p^{-n})\sum_{j\ge-k}p^{-(n+\alpha)j},
\]
and summing gives \eqref{eq:ballenergy}. Similarly,
\[
 \int\rn{\bx}^{-\alpha}\ind_{B_k}(\bx)\dd \bx
 =(1-p^{-n})\sum_{j\ge k}p^{-(n-\alpha)j},
\]
which gives \eqref{eq:ballweight}. For all sufficiently large $M$, the function $u$ equals $u(\boldsymbol0)$ on $B_M$, so $u-u_M=u(\boldsymbol0)\ind_{B_M}$. Set $k=M$ in the two formulas and let $M\to+\infty$. The limits follow from $n-\alpha>0$.
\end{proof}

To prove sharpness, we compute the energy of the critical power function truncated to an annulus.
\begin{proposition}[Energy of an annular truncation]\label{prop:L2sharp}
Let $N\ge1$ and set
\[
 A_N=\bigcup_{k=-N}^N S_k,
 \qquad u_N=h\ind_{A_N}.
\]
Then $u_N\in\Sspace$ and
\begin{equation}\label{eq:L2weight}
 \W_{\alpha/2,2}[u_N]
 =\int\rn{\bx}^{-\alpha}|u_N(\bx)|^2\dd \bx
 =(2N+1)(1-p^{-n}).
\end{equation}
Moreover,
\begin{equation}\label{eq:L2RNexact}
 \E_\alpha[u_N]-\csharp\W_{\alpha/2,2}[u_N]
 =2C_{n,\alpha}(1-p^{-n})^2
 \frac{p^{-(n+\alpha)/2}\bigl(1-p^{-(n+\alpha)(2N+1)/2}\bigr)}
 {(1-p^{-(n+\alpha)/2})^2}.
\end{equation}
In particular, $\E_\alpha[u_N]/\W_{\alpha/2,2}[u_N]\to\csharp$.
\end{proposition}
\begin{proof}
The annulus $A_N$ is compact and open and avoids the origin, so $u_N$ is a test function. Each $S_k\subset A_N$ contributes the same amount to the Hardy integral:
\[
 (1-p^{-n})p^{-nk}p^{\alpha k}p^{(n-\alpha)k}=1-p^{-n},
\]
which gives \eqref{eq:L2weight}.

Substitute $u_N/h=\ind_{A_N}$ into \eqref{eq:L2identity}. By symmetry, the remainder is
\[
 C_{n,\alpha}\int_{A_N}\int_{A_N^c}
 \frac{h(\bx)h(\by)}{\rn{\bx-\by}^{n+\alpha}}\dd \by\dd \bx.
\]
For $\bx\in S_k$, $\by\in S_\ell$, and $k\ne\ell$, the integral over this shell pair is
\begin{align*}
 (1-p^{-n})^2p^{-n(k+\ell)+(n-\alpha)(k+\ell)/2+(n+\alpha)\min(k,\ell)}=(1-p^{-n})^2p^{-(n+\alpha)|k-\ell|/2}.
\end{align*}
Summing over the outer boundary, where $\ell<-N\le k\le N$, gives
\begin{align*}
 \sum_{k=-N}^N\sum_{\ell<-N}p^{-(n+\alpha)(k-\ell)/2}
 &=\sum_{j=0}^{2N}\sum_{m=1}^{\infty}p^{-(n+\alpha)(j+m)/2}\\
 &=\frac{p^{-(n+\alpha)/2}}{1-p^{-(n+\alpha)/2}}
   \sum_{j=0}^{2N}p^{-(n+\alpha)j/2}\\
 &=\frac{p^{-(n+\alpha)/2}\bigl(1-p^{-(n+\alpha)(2N+1)/2}\bigr)}
 {(1-p^{-(n+\alpha)/2})^2}.
\end{align*}
Here $k=-N+j$ and $\ell=-N-m$. The map $(k,\ell)\mapsto(-k,-\ell)$ takes the inner boundary $-N\le k\le N<\ell$ to the outer boundary, preserving $|k-\ell|$. The equal contributions give \eqref{eq:L2RNexact}. The remainder is bounded independently of $N$, while $\W_{\alpha/2,2}[u_N]=(2N+1)(1-p^{-n})\to\infty$, so the quotient tends to $\csharp$.
\end{proof}

Proposition~\ref{prop:L2sharp} proves that $\csharp$ is optimal and, with Corollary~\ref{cor:L2Hardy}, completes the proof of Theorem~\ref{thm:mainL2}. It also gives
\[
 0\le\frac{\E_\alpha[u_N]}{\W_{\alpha/2,2}[u_N]}-\csharp
 \le\frac{2C_{n,\alpha}(1-p^{-n})p^{-(n+\alpha)/2}}
 {(2N+1)(1-p^{-(n+\alpha)/2})^2}.
\]

\subsection{The critical coupling for a singular Schr\"odinger form}
\begin{corollary}[Critical coupling constant]\label{cor:Schrodinger}
Let $0<\alpha<n$ and $b\in\mathbb R$, and define on $\Sspace$
\begin{equation*}
 Q_b[u]=\E_\alpha[u]+b\int\rn{\bx}^{-\alpha}|u(\bx)|^2\dd \bx.
\end{equation*}
Then $Q_b[u]\ge0$ for every $u\in\Sspace$ if and only if
\begin{equation}\label{eq:criticalb}
 b\ge-\csharp.
\end{equation}
If $b<-\csharp$, then
\[
 \inf\{Q_b[u]:u\in\Sspace,\ \|u\|_2=1\}=-\infty.
\]
\end{corollary}
\begin{proof}
If $b\ge-\csharp$, Theorem~\ref{thm:mainL2} gives
\[
 Q_b[u]\ge(\csharp+b)\int\rn{\bx}^{-\alpha}|u|^2\dd \bx\ge0.
\]
If $b<-\csharp$, the sequence in Proposition~\ref{prop:L2sharp} satisfies
\[
 \frac{Q_b[u_N]}{\W_{\alpha/2,2}[u_N]}
 =\frac{\E_\alpha[u_N]}{\W_{\alpha/2,2}[u_N]}+b
 \longrightarrow\csharp+b<0.
\]
Choosing one of these truncations and normalizing it yields $v\in\Sspace$ with $\|v\|_2=1$ and $Q_b[v]<0$.

For $a\in\Qp\setminus\{0\}$, define $v_a(\bx)=|a|_p^{n/2}v(a\bx)$. A change of variables gives
\[
 \|v_a\|_2=1,\qquad
 \E_\alpha[v_a]=|a|_p^\alpha\E_\alpha[v],\qquad
 \int\rn{\bx}^{-\alpha}|v_a(\bx)|^2\dd \bx
 =|a|_p^\alpha\int\rn{\bx}^{-\alpha}|v(\bx)|^2\dd \bx.
\]
Hence $Q_b[v_a]=|a|_p^\alpha Q_b[v]$. Taking $a=p^{-j}$, we have $|a|_p=p^j\to\infty$, and therefore $Q_b[v_a]\to-\infty$.
\end{proof}

For $n=1$, the critical value in \eqref{eq:criticalb} is
\[
 b_*=-\left[\Gamma_p\left(\frac{1+\alpha}{2}\right)\right]^2,
\]
which agrees with the one-dimensional critical potential coefficient in \cite{BGM}. Corollary~\ref{cor:Schrodinger} concerns quadratic forms on the test-function space.

\section{The sharp linear \texorpdfstring{$L^q$}{Lq} Hardy inequality}\label{sec:schur}
We prove Theorem~\ref{thm:mainlinear} under the assumptions \eqref{eq:qrange}.

\subsection{A weighted Riesz potential estimate}
Define
\begin{equation*}
 K_s(\bx,\by)=\rn{\bx}^{-s}k_s(\bx-\by),\qquad
 T_sf(\bx)=\int_{\Qpn}K_s(\bx,\by)f(\by)\dd \by.
\end{equation*}
The kernel $K_s(\bx,\by)$ is strictly positive for almost every $(\bx,\by)$. For nonnegative input, we first allow $T_sf$ to take the value infinity; the following Schur estimate proves $L^q$ boundedness.

\begin{lemma}[Weighted Schur test]\label{lem:Schur}
Let $(X,\mu)$ be a $\sigma$-finite measure space, let $K$ be a nonnegative measurable kernel, and set
\[
 Tf(x)=\int_X K(x,y)f(y)\dd\mu(y),\qquad
 T^*g(y)=\int_X K(x,y)g(x)\dd\mu(x).
\]
If a function $h$, positive and finite almost everywhere, and constants $A,B>0$ satisfy
\[
 T(h^{q'})(x)\le Ah(x)^{q'},\qquad
 T^*(h^q)(y)\le Bh(y)^q,
\]
then $T$ defines a bounded operator on $L^q$, with
\begin{equation*}
 \|T\|_{q\to q}\le A^{1/q'}B^{1/q}.
\end{equation*}
\end{lemma}
\begin{proof}
First suppose $f\ge0$. Write $f(y)=[f(y)h(y)^{-1}]h(y)$ and apply H\"older's inequality with exponents $q$ and $q'$ in the measure $K(x,y)\dd\mu(y)$. This gives
\begin{align*}
 |Tf(x)|^q
 &\le\left(\int K(x,y)h(y)^{q'}\dd\mu(y)\right)^{q-1}
       \int K(x,y)|f(y)|^qh(y)^{-q}\dd\mu(y)\\
 &\le A^{q-1}h(x)^q\int K(x,y)|f(y)|^qh(y)^{-q}\dd\mu(y).
\end{align*}
Since $q'(q-1)=q$, integration in $x$, Tonelli's theorem, and the adjoint bound give
\begin{align*}
 \|Tf\|_q^q
 &\le A^{q-1}\int |f(y)|^qh(y)^{-q}
       \left(\int K(x,y)h(x)^q\dd\mu(x)\right)\dd\mu(y)\\
 &\le A^{q-1}B\|f\|_q^q.
\end{align*}
Apply the calculation first to truncated kernels and nonnegative simple functions with supports of finite measure, then use monotone convergence for nonnegative $L^q$ functions. For complex $f$, the inequality $|Tf|\le T|f|$ gives absolute convergence almost everywhere and the same bound.
\end{proof}

\begin{proposition}[Choice of Schur weights]\label{prop:Schurweights}
Let
\begin{equation*}
 h_0(\bx)=\rn{\bx}^{-n/(qq')}.
\end{equation*}
Then
\begin{equation}\label{eq:eigenweights}
 T_s(h_0^{q'})=\Msharp^{-1}h_0^{q'},\qquad
 T_s^*(h_0^q)=\Msharp^{-1}h_0^q.
\end{equation}
Consequently, $\|T_s\|_{q\to q}\le\Msharp^{-1}$.
\end{proposition}
\begin{proof}
We have $h_0^{q'}(\bx)=\rn{\bx}^{-n/q}=\GammaN(n/q')k_{n/q'}(\bx)$. Since $n/q'+s<n$, Lemma~\ref{lem:semigroup} applies to $k_s*k_{n/q'}$ and gives
\begin{align*}
 I_s(h_0^{q'})(\bx)
 &=\GammaN(n/q')(k_s*k_{n/q'})(\bx)\\
 &=\GammaN(n/q')k_{n/q'+s}(\bx)\\
 &=\frac{\GammaN(n/q')}{\GammaN(n/q'+s)}\rn{\bx}^{s-n/q}.
\end{align*}
Multiplication by $\rn{\bx}^{-s}$ and the definition of $\Msharp$ in \eqref{eq:mainlinear} yield the first identity.

By symmetry of the Riesz kernel,
\[
 T_s^*(h_0^q)(\by)=I_s(\rn{\cdot}^{-s}h_0^q)(\by).
\]
Since $\rn{\bx}^{-s}h_0(\bx)^q=\rn{\bx}^{(n/q-s)-n}$ and $n/q-s>0$, another application of the semigroup identity gives
\begin{align*}
 T_s^*(h_0^q)(\by)
 &=\GammaN(n/q-s)(k_s*k_{n/q-s})(\by)\\
 &=\frac{\GammaN(n/q-s)}{\GammaN(n/q)}\rn{\by}^{-n/q'}\\
 &=\frac{\GammaN(n/q-s)}{\GammaN(n/q)}h_0(\by)^q.
\end{align*}
The reflection identity \eqref{eq:reflection} yields
\[
 \frac{\GammaN(n/q-s)}{\GammaN(n/q)}
 =\frac{\GammaN(n/q')}{\GammaN(n/q'+s)}=\Msharp^{-1}.
\]
Apply Lemma~\ref{lem:Schur} with $h=h_0$ and $A=B=\Msharp^{-1}$.
\end{proof}

\begin{remark}
Since $h_0^{qq'}=\rn{\bx}^{-n}$, we have $h_0^{q'}\notin L^q$ and $h_0^q\notin L^{q'}$. This is allowed in Lemma~\ref{lem:Schur}: the kernel integrals converge absolutely away from the origin and satisfy \eqref{eq:eigenweights}.
\end{remark}

\begin{corollary}[Linear Hardy inequality]\label{cor:linearHardy}
Inequality \eqref{eq:mainlinear} holds for every $u\in\Sspace$.
\end{corollary}
\begin{proof}
Take $f=D^su$. Proposition~\ref{prop:integral} gives $f\in L^q$, and Lemma~\ref{lem:inverse} gives $I_sf=u$ almost everywhere. Therefore
\[
 \W_{s,q}[u]^{1/q}=\|T_sf\|_q
 \le\Msharp^{-1}\|f\|_q=\Msharp^{-1}\|D^su\|_q.
\]
Raising both sides to the power $q$ proves the claim.
\end{proof}

\subsection{The critical power function and the truncation error}\label{sec:linearsharp}
Set
\begin{equation*}
 w(\bx)=\rn{\bx}^{-(n/q-s)}.
\end{equation*}
Since $0<n/q-s<n-s$, Lemma~\ref{lem:power} gives
\begin{equation}\label{eq:linearground}
 D^sw(\bx)
 =\frac{\GammaN(n/q'+s)}{\GammaN(n/q')}\rn{\bx}^{-s}w(\bx)
 =\Msharp\rn{\bx}^{-s}w(\bx).
\end{equation}
The function $w$ satisfies \eqref{eq:linearground} but has infinite Hardy integral. Its annular truncations form a minimizing sequence.

Let $A_N$ be as in \eqref{eq:criticalseqintro}, put $u_N=w\ind_{A_N}$, and define
\begin{equation}\label{eq:linearerror}
 R_N=D^su_N-\Msharp\rn{\bx}^{-s}u_N.
\end{equation}
The identity $sq+q(n/q-s)=n$ gives
\begin{equation}\label{eq:linearweight}
 \W_{s,q}[u_N]=(2N+1)(1-p^{-n}).
\end{equation}

\begin{lemma}[Shell formula for the truncation error]\label{lem:exacterror}
For $\bx\in S_k$,
\begin{equation}\label{eq:exacterror}
 \frac{R_N(\bx)}{C_{n,s}(1-p^{-n})}=
 \begin{cases}
 -\displaystyle p^{(n+s)k}\sum_{\ell=-N}^{N}p^{-(n/q'+s)\ell},&k<-N,\\[5pt]
 \displaystyle p^{(n+s)k}\sum_{\ell>N}p^{-(n/q'+s)\ell}
       +\sum_{\ell<-N}p^{n\ell/q},&-N\le k\le N,\\[5pt]
 -\displaystyle\sum_{\ell=-N}^{N}p^{n\ell/q},&k>N.
 \end{cases}
\end{equation}
In particular, there is a constant $C$, depending only on $p,n,s,q$, such that
\begin{equation}\label{eq:errorbounds}
 |R_N(\bx)|\le C
 \begin{cases}
 p^{(n/q'+s)N}p^{(n+s)k},&k<-N,\\
 p^{-(n/q'+s)N}p^{(n+s)k}+p^{-nN/q},&-N\le k\le N,\\
 p^{nN/q},&k>N.
 \end{cases}
\end{equation}
\end{lemma}
\begin{proof}
Fix $\bx\in S_k$. For $\ell\ne k$, the identity $w|_{S_\ell}=p^{(n/q-s)\ell}$ and \eqref{eq:ultradistance} give
\begin{align*}
 \int_{S_\ell}\frac{w(\by)}{\rn{\bx-\by}^{n+s}}\dd\by
 &=(1-p^{-n})p^{-n\ell}p^{(n/q-s)\ell}p^{(n+s)\min\{k,\ell\}}\\
 &=(1-p^{-n})p^{-(n/q'+s)\ell+(n+s)\min\{k,\ell\}}.
\end{align*}
First suppose $-N\le k\le N$, so that $u_N(\bx)=w(\bx)$. By the integral representation and \eqref{eq:linearground},
\[
 R_N(\bx)=C_{n,s}\int_{A_N^c}\frac{w(\by)}{\rn{\bx-\by}^{n+s}}\dd \by.
\]
For $\ell>N$, the shell radius is smaller than $p^{-k}$, so its distance from $\bx$ is $p^{-k}$ and its contribution is $(1-p^{-n})p^{-(n/q'+s)\ell}p^{(n+s)k}$. For $\ell<-N$, the distance is $p^{-\ell}$ and the contribution is $(1-p^{-n})p^{n\ell/q}$. These sums give the middle line of \eqref{eq:exacterror}.

If $k\notin[-N,N]$, then $u_N(\bx)=0$ and
\[
 R_N(\bx)=-C_{n,s}\int_{A_N}\frac{w(\by)}{\rn{\bx-\by}^{n+s}}\dd \by.
\]
For $k<-N$, the distance from $\bx$ to $S_\ell\subset A_N$ is $p^{-k}$; for $k>N$, it is $p^{-\ell}$. Summation gives the first and third lines of \eqref{eq:exacterror}.

Finally, the geometric-series formulas give
\begin{gather*}
 \sum_{\ell>N}p^{-(n/q'+s)\ell}
 =\frac{p^{-(n/q'+s)N}}{p^{n/q'+s}-1},\qquad
 \sum_{\ell<-N}p^{n\ell/q}
 =\frac{p^{-nN/q}}{p^{n/q}-1},\\
 \sum_{\ell=-N}^{N}p^{-(n/q'+s)\ell}
 \le\frac{p^{(n/q'+s)N}}{1-p^{-(n/q'+s)}},\qquad
 \sum_{\ell=-N}^{N}p^{n\ell/q}
 \le\frac{p^{nN/q}}{1-p^{-n/q}}.
\end{gather*}
Substitution into \eqref{eq:exacterror} proves \eqref{eq:errorbounds}.
\end{proof}

\begin{lemma}[Uniform $L^q$ bound for the error]\label{lem:uniformerror}
There is a constant $C=C(p,n,s,q)$ such that
\begin{equation}\label{eq:uniformerror}
 \|R_N\|_{L^q}\le C\qquad(N\ge1).
\end{equation}
\end{lemma}
\begin{proof}
Since $q'=q/(q-1)$,
\begin{equation*}
 q(n/q'+s)=q(n+s)-n>0.
\end{equation*}
Apply \eqref{eq:errorbounds} on the three regions. The shell measure factor $p^{-nk}$ combines with $p^{q(n+s)k}$ from the $q$th power of the error bound to give $p^{q(n/q'+s)k}$. Factors independent of $N$ are absorbed into $C$.

On the outer region, set $k=-N-j$. Then
\begin{align*}
 \int_{\Qpn\setminus B_{-N}}|R_N(\bx)|^q\dd\bx
 &\le C p^{q(n/q'+s)N}\sum_{k<-N}p^{q(n/q'+s)k}\\
 &=C\sum_{j=1}^{\infty}p^{-q(n/q'+s)j}\\
 &=\frac{C}{p^{q(n/q'+s)}-1}.
\end{align*}
On $A_N$, use $(A+B)^q\le2^{q-1}(A^q+B^q)$ and set $j=N-k$ in the first sum and $j=k+N$ in the second. This gives
\begin{align*}
 \int_{A_N}|R_N(\bx)|^q\dd\bx
 &\le C\left(p^{-q(n/q'+s)N}\sum_{k=-N}^{N}p^{q(n/q'+s)k}
       +p^{-nN}\sum_{k=-N}^{N}p^{-nk}\right)\\
 &=C\sum_{j=0}^{2N}\left(p^{-q(n/q'+s)j}+p^{-nj}\right)\\
 &\le C\left(\frac1{1-p^{-q(n/q'+s)}}+\frac1{1-p^{-n}}\right).
\end{align*}
On the inner region, set $k=N+j$ to obtain
\[
 \int_{B_{N+1}}|R_N(\bx)|^q\dd\bx
 \le C p^{nN}\sum_{k>N}p^{-nk}
 =C\sum_{j=1}^{\infty}p^{-nj}=\frac{C}{p^n-1}.
\]
The three regions partition $\Qpn$ and the bounds are uniform in $N$, proving \eqref{eq:uniformerror}.
\end{proof}

\begin{proposition}[Sharpness of the linear constant]\label{prop:linearsharp}
For the sequence $u_N$ defined above,
\begin{equation}\label{eq:linearratio}
 \frac{\|D^su_N\|_q}{\W_{s,q}[u_N]^{1/q}}
 \longrightarrow\Msharp.
\end{equation}
Consequently, the constant in \eqref{eq:mainlinear} is sharp, and the operator norm is given by \eqref{eq:mainnorm}.
\end{proposition}
\begin{proof}
By \eqref{eq:linearerror} and the reverse triangle inequality,
\[
 \left|\|D^su_N\|_q-\Msharp\W_{s,q}[u_N]^{1/q}\right|
 \le\|R_N\|_q\le C.
\]
Divide by $\W_{s,q}[u_N]^{1/q}=\bigl((2N+1)(1-p^{-n})\bigr)^{1/q}$ and let $N\to\infty$ to obtain \eqref{eq:linearratio}. Together with Corollary~\ref{cor:linearHardy}, this proves that $\Msharp^q$ is the sharp constant in \eqref{eq:mainlinear}.

Set $f_N=D^su_N$. Lemma~\ref{lem:inverse} gives $T_sf_N=\rn{\bx}^{-s}u_N$, so
\[
 \frac{\|T_sf_N\|_q}{\|f_N\|_q}
 =\frac{\W_{s,q}[u_N]^{1/q}}{\|D^su_N\|_q}
 \longrightarrow\Msharp^{-1}.
\]
Combining this limit with the upper bound from Proposition~\ref{prop:Schurweights} proves \eqref{eq:mainnorm}.
\end{proof}

\begin{corollary}[An alternative expression for the linear constant]\label{cor:alphaparam}
Let $\alpha=sq\in(0,n)$. Then
\begin{equation}\label{eq:alphaM}
 M_{p,n,\alpha/q,q}
 =\frac{(1-p^{-(n-\alpha)/q})(1-p^{-n(q-1)/q})}
 {(1-p^{-[n(q-1)+\alpha]/q})(1-p^{-n/q})}.
\end{equation}
Thus, with $\alpha$ as the weight exponent, the sharp constant in the linear Hardy inequality is $M_{p,n,\alpha/q,q}^q$.
\end{corollary}
\begin{proof}
By \eqref{eq:Gammadef},
\[
 \frac{\GammaN(n/q'+s)}{\GammaN(n/q')}
 =\frac{1-p^{s-n/q}}{1-p^{-n/q'-s}}
  \frac{1-p^{-n/q'}}{1-p^{-n/q}}.
\]
Substitute $s=\alpha/q$ and $n/q'=n(q-1)/q$ to obtain \eqref{eq:alphaM}.
\end{proof}

\subsection{Equality and nonattainment of the operator norm}
\begin{proposition}[Equality case]\label{prop:nonattainlinear}
For $f\in L^q(\Qpn)$, the equality $\|T_sf\|_q=\Msharp^{-1}\|f\|_q$ holds if and only if $f=0$ almost everywhere. Consequently, equality in \eqref{eq:mainlinear} for $u\in\Sspace$ holds if and only if $u\equiv0$.
\end{proposition}
\begin{proof}
Suppose a nonzero $f\in L^q$ satisfies $\|T_sf\|_q=\Msharp^{-1}\|f\|_q$. Since
\[
 \|T_sf\|_q\le\|T_s|f|\|_q\le\Msharp^{-1}\|f\|_q,
\]
the norm bound is also attained by $|f|$. Take $h=h_0$ in the proof of Lemma~\ref{lem:Schur}. Both weight estimates are identities by \eqref{eq:eigenweights}, so
\begin{align*}
 \int\Bigg[\Msharp^{-(q-1)}h_0(\bx)^q
       \int K_s(\bx,\by)|f(\by)|^qh_0(\by)^{-q}\dd\by
       -(T_s|f|(\bx))^q\Bigg]\dd\bx
 =\Msharp^{-q}\|f\|_q^q-\|T_s|f|\|_q^q=0.
\end{align*}
The bracket is nonnegative by H\"older's inequality and hence vanishes for almost every $\bx$.

Fix a nonzero point $\bx$ at which this remainder vanishes and all the relevant integrals are finite. Since $K_s(\bx,\by)>0$ for almost every $\by$, the equality case in H\"older's inequality gives a constant $a_{\bx}>0$ such that
\[
 |f(\by)|^q h_0(\by)^{-q}=a_{\bx}h_0(\by)^{q'}
 \quad\text{for almost every }\by.
\]
Using $(q+q')/q=q'$, we obtain
\[
 |f(\by)|=a_{\bx}^{1/q}h_0(\by)^{q'}=a_{\bx}^{1/q}\rn{\by}^{-n/q}.
\]
The constant is positive since $f\ne0$. But
\[
 \int\bigl(\rn{\by}^{-n/q}\bigr)^q\dd \by
 =\int\rn{\by}^{-n}\dd \by=\infty,
\]
which contradicts $f\in L^q$.

Now suppose $u\in\Sspace$ attains equality in \eqref{eq:mainlinear}, and set $f=D^su$. Inversion gives $T_sf=\rn{\bx}^{-s}u$, and hence $\|T_sf\|_q=\Msharp^{-1}\|f\|_q$. Thus $f=0$ almost everywhere. Since $u=I_sf$, we obtain $u=0$ almost everywhere and, by local constancy, $u\equiv0$. The zero function attains equality in both estimates.
\end{proof}

Corollary~\ref{cor:linearHardy} and Propositions~\ref{prop:linearsharp} and \ref{prop:nonattainlinear} complete the proof of Theorem~\ref{thm:mainlinear}.

\section{The sharp Hardy inequality for the nonlinear difference energy}\label{sec:nonlinear}
We prove Theorem~\ref{thm:mainnonlinear} under the assumptions \eqref{eq:qrange}. As before, let $w(\bx)=\rn{\bx}^{-(n/q-s)}$. For $t\in\mathbb R$, write
\begin{equation*}
 J_q(t)=|t|^{q-2}t\quad(t\ne0),\qquad J_q(0)=0.
\end{equation*}
For the critical power function $w$, define
\begin{equation}\label{eq:Lnonlinear}
 \mathscr L_{s,q}w(\bx)
 =2\int_{\Qpn}\frac{J_q(w(\bx)-w(\by))}{\rn{\bx-\by}^{n+sq}}\dd \by,
 \qquad \bx\ne\boldsymbol0.
\end{equation}
Only its action on $w$ is used below; the next lemma gives absolute convergence.

\subsection{The nonlinear equation for the critical power function}
\begin{lemma}[Absolute integral estimate]\label{lem:nonlinearabs}
For $\bx\ne\boldsymbol0$,
\begin{equation}\label{eq:nonlinearabs}
 \begin{aligned}
 &2\int\frac{|w(\bx)-w(\by)|^{q-1}}{\rn{\bx-\by}^{n+sq}}\dd \by\\
 =&2(1-p^{-n})\rn{\bx}^{-sq}w(\bx)^{q-1}\sum_{m=1}^{\infty}p^{-sqm}
 \bigl(1-p^{-(n/q-s)m}\bigr)^{q-1}\bigl(1+p^{-(n/q-s)m}\bigr).
 \end{aligned}
\end{equation}
The series converges, and
\begin{equation}\label{eq:nonlinearabsbound}
 2\int\frac{|w(\bx)-w(\by)|^{q-1}}{\rn{\bx-\by}^{n+sq}}\dd \by
 \le\frac{4(1-p^{-n})}{p^{sq}-1}\rn{\bx}^{-sq}w(\bx)^{q-1}.
\end{equation}
\end{lemma}
\begin{proof}
If $\rn{\bx}=1$, then $w(\bx)=1$. The distance from $\bx$ to $S_{-m}$, $m\ge1$, is $p^m$, so this shell contributes
\[
 (1-p^{-n})p^{-sqm}\bigl(1-p^{-(n/q-s)m}\bigr)^{q-1}.
\]
The distance to $S_m$ is $1$, and its contribution is
\begin{align*}
 (1-p^{-n})p^{-nm}\bigl(p^{(n/q-s)m}-1\bigr)^{q-1}=(1-p^{-n})p^{-[n-(n/q-s)(q-1)]m}
 \bigl(1-p^{-(n/q-s)m}\bigr)^{q-1}.
\end{align*}
The exponents satisfy
\begin{equation}\label{eq:nonlinearbalance}
 n-(n/q-s)(q-1)=sq+n/q-s.
\end{equation}
Adding the contributions and multiplying by $2$ gives \eqref{eq:nonlinearabs}: the shell containing $\bx$ contributes zero. This proves the identity for $\rn{\bx}=1$. Since $n/q-s>0$,
\[
 \bigl(1-p^{-(n/q-s)m}\bigr)^{q-1}\bigl(1+p^{-(n/q-s)m}\bigr)\le2,
\]
so the $m$th contribution, including the prefactor, is at most $4(1-p^{-n})p^{-sqm}$. Summing over $m$ proves \eqref{eq:nonlinearabsbound}.
For general $\bx\ne\boldsymbol0$, write $\bx=p^k\bx_0$, with $\rn{\bx_0}=1$, and set $\by=p^k\bz$. The scaling identity $w(p^k\bz)=p^{(n/q-s)k}w(\bz)$ gives
\begin{align*}
 \int\frac{|w(\bx)-w(\by)|^{q-1}}{\rn{\bx-\by}^{n+sq}}\dd\by
 &=p^{[sq+(n/q-s)(q-1)]k}
 \int\frac{|1-w(\bz)|^{q-1}}{\rn{\bx_0-\bz}^{n+sq}}\dd\bz.
\end{align*}
The prefactor equals $\rn{\bx}^{-sq}w(\bx)^{q-1}$, giving both formulas.
\end{proof}

\begin{proposition}[Nonlinear equation for the critical power function]\label{prop:nonlinearground}
For every $\bx\ne\boldsymbol0$,
\begin{equation*}
 \mathscr L_{s,q}w(\bx)=\Hsharp\rn{\bx}^{-sq}w(\bx)^{q-1},
\end{equation*}
where $\Hsharp$ is defined by \eqref{eq:Hseries}.
\end{proposition}
\begin{proof}
By Lemma~\ref{lem:nonlinearabs}, the integral converges absolutely. For $\rn{\bx}=1$, the outer shell $S_{-m}$ contributes
\[
 (1-p^{-n})p^{-sqm}\bigl(1-p^{-(n/q-s)m}\bigr)^{q-1};
\]
and the inner shell $S_m$ contributes
\begin{align*}
 -(1-p^{-n})p^{-nm}\bigl(p^{(n/q-s)m}-1\bigr)^{q-1}=-(1-p^{-n})p^{-(sq+n/q-s)m}
 \bigl(1-p^{-(n/q-s)m}\bigr)^{q-1}.
\end{align*}
By \eqref{eq:nonlinearbalance}, their sum is
\[
 (1-p^{-n})p^{-sqm}\bigl(1-p^{-(n/q-s)m}\bigr)^q.
\]
Summing over $m\ge1$ and multiplying by the factor $2$ in \eqref{eq:Lnonlinear} proves the equation on the unit shell. The homogeneity $J_q(ct)=c^{q-1}J_q(t)$ for $c>0$ gives the same scaling factor $\rn{\bx}^{-sq}w(\bx)^{q-1}$ as in Lemma~\ref{lem:nonlinearabs}, proving the general case.
\end{proof}

\subsection{Picone's inequality and the remainder identity}
\begin{lemma}[Two-point Picone inequality]\label{lem:Picone}
Let $q>1$, $a,b\ge0$, and $r,t>0$. Then
\begin{equation*}
 |a-b|^q\ge J_q(r-t)
 \left(\frac{a^q}{r^{q-1}}-\frac{b^q}{t^{q-1}}\right).
\end{equation*}
Equality holds if and only if $a/r=b/t$.
\end{lemma}
\begin{proof}
For $r=t$, the assertion reduces to nonnegativity of the left-hand side, with equality precisely when $a=b$, equivalently $a/r=b/t$. Suppose $r>t$. Then $0<t/r<1$ and
\[
 a=\left(1-\frac tr\right)\frac{a-b}{1-t/r}+\frac tr\,\frac{b}{t/r}.
\]
Strict convexity of the real function $z\mapsto|z|^q$ gives
\[
 a^q\le\left(1-\frac tr\right)^{1-q}|a-b|^q
       +\left(\frac tr\right)^{1-q}b^q.
\]
Multiplying by $(1-t/r)^{q-1}$ and rearranging, we obtain
\begin{align*}
 |a-b|^q
 &\ge\left(1-\frac tr\right)^{q-1}a^q
 -\left(1-\frac tr\right)^{q-1}\left(\frac tr\right)^{1-q}b^q\\
 &=(r-t)^{q-1}\left(\frac{a^q}{r^{q-1}}-\frac{b^q}{t^{q-1}}\right).
\end{align*}
Equality in the convexity inequality holds exactly when
\[
 \frac{a-b}{1-t/r}=\frac b{t/r},
\]
which is equivalent to $a/r=b/t$. For $r<t$, interchange $(a,r)$ and $(b,t)$; both factors on the right change sign, so the product is unchanged.
\end{proof}

\begin{definition}[Picone remainder]\label{def:defect}
For a complex-valued measurable function $u$ and $\bx,\by\ne\boldsymbol0$, define
\begin{align}
 \mathfrak D_w[u](\bx,\by)
 ={}&|u(\bx)-u(\by)|^q\notag\\
 &-J_q(w(\bx)-w(\by))
 \left(\frac{|u(\bx)|^q}{w(\bx)^{q-1}}
       -\frac{|u(\by)|^q}{w(\by)^{q-1}}\right).
 \label{eq:defect}
\end{align}
The reverse triangle inequality $\big||u(\bx)|-|u(\by)|\big|\le|u(\bx)-u(\by)|$ and Lemma~\ref{lem:Picone} imply that $\mathfrak D_w[u](\bx,\by)\ge0$.
\end{definition}

\begin{theorem}[Nonlinear ground state representation]\label{thm:nonlinearidentity}
For every $u\in\X$,
\begin{equation}\label{eq:nonlinearidentity}
 \G_{s,q}[u]-\Hsharp\W_{s,q}[u]
 =\iint\frac{\mathfrak D_w[u](\bx,\by)}{\rn{\bx-\by}^{n+sq}}\dd \bx\dd \by\ge0.
\end{equation}
The remainder integral is finite.
\end{theorem}
\begin{proof}
By Lemma~\ref{lem:nonlinearabs} and Tonelli's theorem,
\begin{align*}
 \iint\frac{|J_q(w(\bx)-w(\by))|}{\rn{\bx-\by}^{n+sq}}
 \left|\frac{|u(\bx)|^q}{w(\bx)^{q-1}}-\frac{|u(\by)|^q}{w(\by)^{q-1}}\right|
 \dd \bx\dd \by&\le2\int\frac{|u(\bx)|^q}{w(\bx)^{q-1}}
 \left(\int\frac{|w(\bx)-w(\by)|^{q-1}}{\rn{\bx-\by}^{n+sq}}\dd \by\right)\dd \bx\\
 &\le\frac{4(1-p^{-n})}{p^{sq}-1}\W_{s,q}[u]\\
 &<\infty.
\end{align*}
Thus the signed term in \eqref{eq:defect}, multiplied by the kernel, is absolutely integrable. Symmetry of the kernel and antisymmetry of $J_q(w(\bx)-w(\by))$ give
\begin{align*}
 \iint\frac{J_q(w(\bx)-w(\by))}{\rn{\bx-\by}^{n+sq}}
 \left(\frac{|u(\bx)|^q}{w(\bx)^{q-1}}-\frac{|u(\by)|^q}{w(\by)^{q-1}}\right)
 \dd \bx\dd \by&=2\int\frac{|u(\bx)|^q}{w(\bx)^{q-1}}
 \left(\int\frac{J_q(w(\bx)-w(\by))}{\rn{\bx-\by}^{n+sq}}\dd \by\right)\dd \bx\\
 &=\int\frac{|u(\bx)|^q}{w(\bx)^{q-1}}\mathscr L_{s,q}w(\bx)\dd \bx\\
 &=\Hsharp\W_{s,q}[u],
\end{align*}
where the last step uses Proposition~\ref{prop:nonlinearground}. Since $\G_{s,q}[u]<\infty$, integration of \eqref{eq:defect} proves \eqref{eq:nonlinearidentity}. Nonnegativity of the remainder follows from Definition~\ref{def:defect}.
\end{proof}

\begin{corollary}[Hardy inequality and equality case]\label{cor:nonlinearHardy}
Inequality \eqref{eq:mainnonlinear} holds for every $u\in\X$, with equality if and only if $u=0$ almost everywhere. In particular, equality in $\Sspace$ holds only for $u\equiv0$.
\end{corollary}
\begin{proof}
The inequality follows from Theorem~\ref{thm:nonlinearidentity}, and Lemma~\ref{lem:localconstant} gives $\Sspace\subset\X$. If equality holds, then $\mathfrak D_w[u](\bx,\by)=0$ for almost every $(\bx,\by)$. For such a pair,
\begin{align*}
 |u(\bx)-u(\by)|^q
 &\ge\bigl||u(\bx)|-|u(\by)|\bigr|^q\\
 &\ge J_q(w(\bx)-w(\by))
 \left(\frac{|u(\bx)|^q}{w(\bx)^{q-1}}
       -\frac{|u(\by)|^q}{w(\by)^{q-1}}\right).
\end{align*}
By \eqref{eq:defect}, equality holds throughout, in particular in Picone's inequality. Lemma~\ref{lem:Picone} gives
\[
 \frac{|u(\bx)|}{w(\bx)}=\frac{|u(\by)|}{w(\by)}
 \quad\text{for almost every }(\bx,\by).
\]
Fubini's theorem now implies $|u|=cw$ almost everywhere for some $c\ge0$. If $c>0$, then
\[
 \W_{s,q}[u]=c^q\int\rn{\bx}^{-sq-q(n/q-s)}\dd \bx
 =c^q(1-p^{-n})\sum_{k\in\mathbb Z}1=\infty,
\]
contrary to $u\in\X$. Thus $c=0$ and $u=0$ almost everywhere. Conversely, the zero function attains equality.
\end{proof}

\begin{remark}\label{rem:spacewarning}
In \eqref{eq:Xdef}, both the difference energy and the Hardy term are assumed finite. A nonzero constant has zero difference energy but infinite Hardy integral when $sq<n$. A formulation of \eqref{eq:mainnonlinear} based on the difference energy alone therefore needs an additional condition at infinity.
\end{remark}

\subsection{Shell counting and sharpness}\label{sec:nonlinearsharp}
Again let $u_N=w\ind_{A_N}$. By \eqref{eq:linearweight},
\[
 \W_{s,q}[u_N]=(2N+1)(1-p^{-n}).
\]

\begin{lemma}[Contribution of a shell pair]\label{lem:shellpairs}
Let $k<\ell$. The contribution of $(S_k\times S_\ell)\cup(S_\ell\times S_k)$ to $\G_{s,q}[u_N]$ is
\begin{equation*}
 \begin{cases}
 2(1-p^{-n})^2p^{-sq(\ell-k)}
 \bigl(1-p^{-(n/q-s)(\ell-k)}\bigr)^q,&-N\le k<\ell\le N,\\
 2(1-p^{-n})^2p^{-sq(\ell-k)},&k<-N\le\ell\le N,\\
 2(1-p^{-n})^2p^{-n(\ell-k)},&-N\le k\le N<\ell,\\
 0,&k,\ell\notin[-N,N].
 \end{cases}
\end{equation*}
Pairs of points on the same shell contribute zero.
\end{lemma}
\begin{proof}
For $-N\le j\le N$, the function $u_N$ equals $p^{(n/q-s)j}$ on $S_j$; on all other shells it vanishes. Since $k<\ell$, the distance between the two shells is $p^{-k}$, and hence
\begin{equation*}
 \iint_{(S_k\times S_\ell)\cup(S_\ell\times S_k)}
 \frac{\dd\bx\dd\by}{\rn{\bx-\by}^{n+sq}}
 =2(1-p^{-n})^2p^{sqk-n\ell}.
\end{equation*}
Multiply by the $q$th power of the absolute difference in function values.

If both shells lie in $A_N$, set $m=\ell-k$. Since $q(n/q-s)=n-sq$, the contribution is
\begin{align*}
 2(1-p^{-n})^2p^{sqk-n\ell}p^{(n-sq)\ell}
 \bigl(1-p^{-(n/q-s)m}\bigr)^q=2(1-p^{-n})^2p^{-sqm}\bigl(1-p^{-(n/q-s)m}\bigr)^q.
\end{align*}
If $k<-N\le\ell\le N$, then $S_k$ lies outside the annulus and $u_N$ vanishes there, while its value on $S_\ell$ is $p^{(n/q-s)\ell}$. The contribution is
\[
 2(1-p^{-n})^2p^{sqk-n\ell}p^{(n-sq)\ell}
 =2(1-p^{-n})^2p^{-sq(\ell-k)}.
\]
If $-N\le k\le N<\ell$, then $S_\ell$ lies on the inner side of the annulus and $u_N$ vanishes there, while its value on $S_k$ is $p^{(n/q-s)k}$. The contribution is
\[
 2(1-p^{-n})^2p^{sqk-n\ell}p^{(n-sq)k}
 =2(1-p^{-n})^2p^{-n(\ell-k)}.
\]
The contribution vanishes if both shells lie outside $A_N$, or if the two points lie on the same shell, where $u_N$ is constant.
\end{proof}

\begin{proposition}[Exact formula for the truncated energy]\label{prop:nonlinearenergy}
For $N\ge1$,
\begin{equation}\label{eq:nonlinearenergy}
 \begin{aligned}
 \G_{s,q}[u_N]=2(1-p^{-n})^2\Bigg[&
 \sum_{m=1}^{2N}(2N+1-m)p^{-sqm}\bigl(1-p^{-(n/q-s)m}\bigr)^q\\
 &+\frac{p^{-n}(1-p^{-n(2N+1)})}{(1-p^{-n})^2}
 +\frac{p^{-sq}(1-p^{-sq(2N+1)})}{(1-p^{-sq})^2}\Bigg].
 \end{aligned}
\end{equation}
\end{proposition}
\begin{proof}
The annulus $A_N$ has $2N+1$ shells. An interior pair $(S_k,S_{k+m})$ with index difference $m$ has $k$ ranging from $-N$ to $N-m$, so there are $2N+1-m$ pairs. Their total contribution is
\[
 2(1-p^{-n})^2\sum_{m=1}^{2N}(2N+1-m)p^{-sqm}
 \bigl(1-p^{-(n/q-s)m}\bigr)^q.
\]
By Lemma~\ref{lem:shellpairs}, the inner boundary contribution corresponds to $-N\le k\le N<\ell$ and equals
\begin{align*}
 2(1-p^{-n})^2\sum_{k=-N}^{N}\sum_{\ell>N}p^{-n(\ell-k)}
 &=2(1-p^{-n})^2\frac{p^{-n}}{1-p^{-n}}\sum_{j=0}^{2N}p^{-nj}\\
 &=2(1-p^{-n})^2\frac{p^{-n}(1-p^{-n(2N+1)})}{(1-p^{-n})^2}.
\end{align*}
The outer boundary contribution corresponds to $k<-N\le\ell\le N$ and equals
\begin{align*}
 2(1-p^{-n})^2\sum_{\ell=-N}^{N}\sum_{k<-N}p^{-sq(\ell-k)}
 &=2(1-p^{-n})^2\frac{p^{-sq}}{1-p^{-sq}}\sum_{j=0}^{2N}p^{-sqj}\\
 &=2(1-p^{-n})^2\frac{p^{-sq}(1-p^{-sq(2N+1)})}{(1-p^{-sq})^2}.
\end{align*}
Adding these three contributions proves \eqref{eq:nonlinearenergy}.
\end{proof}

\begin{theorem}[Sharpness and a convergence estimate]\label{thm:nonlinearsharp}
The constant $\Hsharp$ in \eqref{eq:mainnonlinear} is sharp. Moreover,
\begin{equation}\label{eq:nonlinearrate}
 0\le\frac{\G_{s,q}[u_N]}{\W_{s,q}[u_N]}-\Hsharp
 \le\frac{2(1-p^{-n})}{2N+1}
 \left(\frac{p^{-n}}{(1-p^{-n})^2}+\frac{p^{-sq}}{(1-p^{-sq})^2}\right),
\end{equation}
so the energy quotient converges to $\Hsharp$ with error $O((2N+1)^{-1})$.
\end{theorem}
\begin{proof}
The lower bound is Corollary~\ref{cor:nonlinearHardy}. For the upper bound, nonnegativity of the summands gives
\begin{align*}
 \sum_{m=1}^{2N}(2N+1-m)p^{-sqm}\bigl(1-p^{-(n/q-s)m}\bigr)^q\le(2N+1)\sum_{m=1}^{\infty}p^{-sqm}\bigl(1-p^{-(n/q-s)m}\bigr)^q.
\end{align*}
Substitute into \eqref{eq:nonlinearenergy}, use \eqref{eq:Hseries} and \eqref{eq:linearweight}, and bound $1-p^{-n(2N+1)}$ and $1-p^{-sq(2N+1)}$ by $1$ to obtain \eqref{eq:nonlinearrate}. As $N\to\infty$, the quotient tends to $\Hsharp$, proving sharpness.
\end{proof}

\begin{remark}[Boundedness of the truncation remainder]
Subtracting $\Hsharp\W_{s,q}[u_N]$ from \eqref{eq:nonlinearenergy} gives
\begin{align*}
 \frac{\G_{s,q}[u_N]-\Hsharp\W_{s,q}[u_N]}{2(1-p^{-n})^2}
 ={}&\frac{p^{-n}(1-p^{-n(2N+1)})}{(1-p^{-n})^2}
 +\frac{p^{-sq}(1-p^{-sq(2N+1)})}{(1-p^{-sq})^2}\\
 &-\sum_{m=1}^{2N}mp^{-sqm}\bigl(1-p^{-(n/q-s)m}\bigr)^q\\
 &-(2N+1)\sum_{m=2N+1}^{\infty}p^{-sqm}\bigl(1-p^{-(n/q-s)m}\bigr)^q.
\end{align*}
Since $0\le(1-p^{-(n/q-s)m})^q\le1$,
\[
 \sum_{m=1}^{\infty}mp^{-sqm}\bigl(1-p^{-(n/q-s)m}\bigr)^q
 \le\sum_{m=1}^{\infty}mp^{-sqm}<\infty,
\]
and
\begin{align*}
 (2N+1)\sum_{m=2N+1}^{\infty}p^{-sqm}\bigl(1-p^{-(n/q-s)m}\bigr)^q\le\sum_{m=2N+1}^{\infty}mp^{-sqm}\bigl(1-p^{-(n/q-s)m}\bigr)^q.
\end{align*}
Thus both boundary contributions and both shell-counting errors are uniformly bounded in $N$.
\end{remark}

Corollary~\ref{cor:nonlinearHardy} and Theorem~\ref{thm:nonlinearsharp} complete the proof of Theorem~\ref{thm:mainnonlinear}.

\subsection{Relations for \texorpdfstring{$q=2$}{q=2}}\label{sec:compatibility}
\begin{proposition}[Energies and sharp constants in the quadratic case]\label{prop:q2}
Let $q=2$, $0<s<n/2$, and $\alpha=2s$. For every $u\in\Sspace$,
\begin{equation}\label{eq:q2energy}
 \|D^su\|_2^2=\E_{2s}[u]
 =\frac{C_{n,2s}}2\G_{s,2}[u].
\end{equation}
The sharp constants satisfy
\begin{equation}\label{eq:q2constants}
 M_{p,n,s,2}^2=C^\sharp_{p,n,2s},\qquad
 H_{p,n,s,2}=\frac2{C_{n,2s}}C^\sharp_{p,n,2s}.
\end{equation}
\end{proposition}
\begin{proof}
The Fourier multiplier definition and Plancherel's identity give
\[
 \|D^su\|_2^2=\int\rn{\bxi}^{2s}|\hat u(\bxi)|^2\dd\bxi
 =\E_{2s}[u].
\]
The jump-kernel formula \eqref{eq:quadratic} then yields \eqref{eq:q2energy}. Since $\GammaN(n/2)=1$,
\[
 M_{p,n,s,2}^2
 =[\GammaN(n/2+s)]^2=C^\sharp_{p,n,2s}.
\]
Substituting \eqref{eq:q2energy} into the corresponding inequalities and using sharpness proves the second identity in \eqref{eq:q2constants}. It can also be checked directly from the series: set $q=2$ in \eqref{eq:Hseries}, expand the square, and sum to obtain
\begin{align*}
 H_{p,n,s,2}
 &=2(1-p^{-n})\sum_{m\ge1}
 \left(p^{-2sm}-2p^{-(n/2+s)m}+p^{-nm}\right)\\
 &=2(1-p^{-n})\left(\frac1{p^{2s}-1}
       -\frac2{p^{n/2+s}-1}+\frac1{p^n-1}\right)\\
 &=\frac{2(p^{n/2+s}+1)(p^{n/2}-p^s)^2}
 {p^n(p^{2s}-1)(p^{n/2+s}-1)}\\
 &=\frac{2(1-p^{-n-2s})}{p^{2s}-1}
 \left(\frac{1-p^{-(n/2-s)}}{1-p^{-(n/2+s)}}\right)^2\\
 &=\frac{2}{C_{n,2s}}C^\sharp_{p,n,2s}.
\end{align*}
The simplification uses $p^{2s}p^n=(p^{n/2+s})^2$, followed by \eqref{eq:kernelconstant} and the formula for the sharp quadratic-form constant.
\end{proof}

\begin{remark}[The two $L^q$ energies]
For $q\ne2$, $\|D^su\|_q^q$ and $\G_{s,q}[u]$ are distinct, and \eqref{eq:q2energy} no longer follows from Plancherel's identity. They have the same homogeneity and critical power $w=\rn{\bx}^{-(n/q-s)}$, but their sharp constants come from different arguments: the weighted Riesz potential norm and the nonlinear equation for the critical power, respectively.
\end{remark}

\section*{Acknowledgments}
I am grateful to my supervisor, Professor Bobo Hua, for his invaluable support and guidance.

\bibliographystyle{unsrt}
\bibliography{reference}
\noindent\emph{Yaojia Sun, School of Mathematical Sciences, Fudan University,
Shanghai 200433, China}\\
\noindent\emph{Email:}
\href{mailto:26110180043@m.fudan.edu.cn}{\texttt{26110180043@m.fudan.edu.cn}}
\end{document}